\documentclass[12pt]{article}
\usepackage{lineno,hyperref}
\modulolinenumbers[1]

\usepackage{amsmath}
\usepackage{graphicx}
\usepackage{amsfonts}
\usepackage{float}
\usepackage{verbatim}
\usepackage{amsthm}
\usepackage{amssymb}
\usepackage{latexsym}
\usepackage{epstopdf}
\usepackage{color}
\usepackage{a4}
\usepackage{amsmath}
\allowdisplaybreaks[4]

\makeatletter
\@addtoreset{equation}{section}

\begin{document}

\newcommand{\EE}{\mathbb{E}}
\newcommand{\PP}{\mathbb{P}}
\newcommand{\RR}{\mathbb{R}}
\newcommand{\SM}{\mathbb{S}}
\newcommand{\ZZ}{\mathbb{Z}}
\newcommand{\ind}{\mathbf{1}}
\newcommand{\LL}{\mathbb{L}}
\def\F{{\cal F}}
\def\G{{\cal G}}
\def\P{{\cal P}}

\newtheorem{theorem}{Theorem}[section]
\newtheorem{lemma}[theorem]{Lemma}
\newtheorem{coro}[theorem]{Corollary}
\newtheorem{defn}[theorem]{Definition}
\newtheorem{assp}[theorem]{Assumption}
\newtheorem{cond}[theorem]{Condition}
\newtheorem{expl}[theorem]{Example}
\newtheorem{prop}[theorem]{Proposition}
\newtheorem{rmk}[theorem]{Remark}
\newtheorem{conj}[theorem]{Conjecture}

\newcommand\tq{{\scriptstyle{3\over 4 }\scriptstyle}}
\newcommand\qua{{\scriptstyle{1\over 4 }\scriptstyle}}
\newcommand\hf{{\textstyle{1\over 2 }\displaystyle}}
\newcommand\hhf{{\scriptstyle{1\over 2 }\scriptstyle}}
\newcommand\hei{\tfrac{1}{8}}

\newcommand{\eproof}{\indent\vrule height6pt width4pt depth1pt\hfil\par\medbreak}

\def\a{\alpha}
\def\e{\varepsilon} \def\z{\zeta} \def\y{\eta} \def\o{\theta}
\def\vo{\vartheta} \def\k{\kappa} \def\l{\lambda} \def\m{\mu} \def\n{\nu}
\def\x{\xi}  \def\r{\rho} \def\s{\sigma}
\def\p{\phi} \def\f{\varphi}   \def\w{\omega}
\def\q{\surd} \def\i{\bot} \def\h{\forall} \def\j{\emptyset}

\def\be{\beta} \def\de{\delta} \def\up{\upsilon} \def\eq{\equiv}
\def\ve{\vee} \def\we{\wedge}

\def\D{\Delta} \def\O{\Theta} \def\L{\Lambda}
\def\X{\Xi} \def\Si{\Sigma} \def\W{\Omega}
\def\M{\partial} \def\N{\nabla} \def\Ex{\exists} \def\K{\times}
\def\V{\bigvee} \def\U{\bigwedge}

\def\1{\oslash} \def\2{\oplus} \def\3{\otimes} \def\4{\ominus}
\def\5{\circ} \def\6{\odot} \def\7{\backslash} \def\8{\infty}
\def\9{\bigcap} \def\0{\bigcup} \def\+{\pm} \def\-{\mp}
\def\la{\langle} \def\ra{\rangle}

\def\proof{\noindent{\it Proof. }}
\def\tl{\tilde}
\def\trace{\hbox{\rm trace}}
\def\diag{\hbox{\rm diag}}
\def\for{\quad\hbox{for }}
\def\refer{\hangindent=0.3in\hangafter=1}

\newcommand\wD{\widehat{\D}}
\newcommand{\ka}{\kappa_{10}}

\title{Truncated Milstein method for non-autonomous stochastic differential equations and its modification}

\author{Juan Liao ~~Wei Liu\footnote{Corresponding author, Email: weiliu@shnu.edu.cn; lwbvb@hotmail.com}~~ Xiaoyan Wang\\ Department of Mathematics\\ Shanghai Normal University, Shanghai, 200234, China}
\date{}

\maketitle

\begin{abstract}
The truncated Milstein method, which was initially proposed in (Guo, Liu, Mao and Yue 2018), is extended to the non-autonomous stochastic differential equations with the super-linear state variable and the H\"older continuous time variable. The convergence rate is proved. Compared with the initial work, the requirements on the step-size is also significantly released. In addition, the technique of the randomized step-size is employed to raise the convergence rate of the truncated Milstein method.
\end{abstract}

\section{Introduction}
Numerical methods for stochastic differential equations (SDEs) with super-linear coefficients have been attracting lots of attention in recent years. Due to that the classical Euler-Maruyama (EM) method fails to converge for those types of SDEs \cite{HJK2011}, different new methods have been proposed. 
\par
Implicit methods are natural alternatives, since they have been successful in handling the stiffness in ordinary differential equations (ODEs) \cite{HW2010}. The implicit methods of the Euler's type for SDEs were studied in \cite{BK2010,Hu1996,MS2013,WWD2020,YXT2017}. The Milstein-type implicit methods for SDEs were discussed in  \cite{HMS2013,KS2006,RKVZ2015,ZWX2018}. The multi-stage implicit methods were investigated in \cite{AH2017,BT2004,DR2009}. We just mention some of the works on implicit methods here and refer the readers to the references therein.
\par
Compared with implicit methods, explicit methods still have their advantages, such as  simple structures, easy to implement, cheap to simulate large numbers of paths \cite{Higham2011}. Recently, many different explicit methods have been proposed to approximate SDEs with super-linear coefficients. The tamed Euler method was initially proposed in \cite{HJK2012}. By using the idea of taming the coefficients, different types of tamed methods have been proposed \cite{DKS2016,GHW2020,NL2019,SLL2018,WG2013,ZM2017,ZWH2014}. The truncated EM method is another modification of the classical EM, which was initialized in \cite{Mao2015,Mao2016}. Afterwards, the truncating technique has been employed to develop different kinds of truncated methods \cite{DFLM2019,GLMY2017,JHL2018,LX2018,LMY2019,ZSL2018}.
\par
Most of the works mentioned above dealt with autonomous SDEs, where the time variable does not appear explicitly in the coefficients. Meanwhile, it is well-known that the non-smoothness of the time variable in non-autonomous SDEs leads to significant difference in the convergence rate of numerical methods for both ODEs \cite{Daun2011,Kacewicz1987} and SDEs \cite{KW2019,LMTW2020}. 
\par
Motivated by all the issues mentioned above, we investigate the truncated Milstein method for non-autonomous SDEs with the time variable satisfying H\"older's continuity and the state variable containing super-linear terms. The finite time convergence of the proposed method is proved and the convergence  rate is discussed. This result could be regarded as an extension of \cite{GLMY2018}, where autonomous SDEs were considered. We also propose the randomized truncated Milstein method to overcome the low convergence rate due to the H\"older continuous time variable.
\par
The main contribution of this paper are twofold.
\begin{itemize}
\item Compared with the existing work \cite{GLMY2018}, our results cover the non-autonomous case and release the requirements on the step-size significantly.
\item The randomized truncated Milstein method is proposed. By using numerical silumations, this new method is demonstrated to outperform the truncated Milstein method for non-autonomous SDEs.
\end{itemize}
\par
This paper is constructed in the following way. Notations, assumptions and the structure of the numerical methods are presented in Section 2. Section 3 contains the main results and their proofs. The randomized truncated Milstein method is proposed in Section 4. Numerical simulations are conducted in Section 5. Section 6 sees the conclusion and some discussions on the future research.

\section{Mathematical preliminary} 
Notations, assumptions and the truncated Milstein method for non-autonomous SDEs are introduced in the section.
\par
Through out this paper, unless otherwise specified, we let ($\Omega$, $\mathcal{F}$, $\mathbb{P}$) be a complete probability space with a filtration $\{\mathcal{F}_t\}_{t\geq 0}$ satisfying the usual condition (that is, it is right continuous and increasing while $\mathcal{F}_0$ contains all p-null sets). Let $B(t)$ be an one-dimensional Brownian motion defined in the probability space and is $\F_t$-adopted. And let $\vert \cdot \vert$ denote both 
the Euclidean norm in $\mathbb{R}^n$ and the trace norm in $\mathbb{R}^{n\times m}$; Moreover, for two real numbers a and b, we use $a \lor b=\max(a,b)$ and $a \land b=\min(a,b)$. For a given set G, its indicator function is denoted by $I_G$, namely $I_G(x)=1$ if $x\in G$ and 0 otherwise.
\par
We are concerned with the d-dimension SDEs

\begin{equation}
dy(t)=\mu (t,y(t))dt+\sigma (t,y(t))dB(t)
\label{SDE}
\end{equation}\\
\textcolor{blue}{on $t\in[t_0,T]$ for any $T>t_0$ with the initial value $y(t_0)=y_0\in \mathbb{R}^d$,} where the drift coefficient function $\mu :[t_0,T]\times \mathbb{R}^d\to \mathbb{R}^d
$  and the diffusion coefficient function $\sigma :[t_0,T]\times \mathbb{R}^d\to \mathbb{R}^d
,$ and $y(t)=(y^1(t),y^2(t),...,y^d(t))^T.$\par
We define:
\begin{equation*}
{L\sigma(t,y)=\sum_{l=1}^d\sigma ^l(t,y)\frac {\partial\sigma(t,y)}{ \partial y^l}},
\end{equation*}
where $\sigma=(\sigma^1,\sigma^2,...,\sigma^d)^T,\quad\sigma^l:\mathbb{R}^d\to\mathbb{R}.$\\ 
And define the derivative of vector $\sigma(t,y)$ with respect to $y^l$ by\\
\begin{equation*}
G^l(t,y):=\left(\frac {\partial\sigma^1\left(t,y\right)}{ \partial y^l},\frac {\partial\sigma^2\left(t,y\right)}{ \partial y^l},...,\frac {\partial\sigma^d\left(t,y\right)}{ \partial y^l}\right).
\end{equation*}\\
Moreover, we assume that both $\sigma$ \textcolor{blue}{and} $\mu$ have \textcolor{blue}{a} second-order derivative, and we make the following assumptions.\\
\begin{assp}\label{A1}
There exist constants $C_1>0$ \textcolor{blue}{,} $\beta>0$ \textcolor{blue}{and $\a\in(0,1]$} such that\\
\begin{equation*}
\vert\mu(t,x)-\mu(t,y)\vert\vee\vert\sigma(t,x)-\sigma(t,y)\vert\vee\vert L\sigma(t,x)-L\sigma(t,y)\vert\leq C_1(1+\vert x\vert ^{\beta}+\vert y\vert ^{\beta})\vert x-y\vert,
\end{equation*}
\begin{equation*}
\textcolor{blue}{\vert\mu(t_1,y)-\mu(t_2,y)\vert\vee\vert\sigma(t_1,y)-\sigma(t_2,y)\vert\leq  C_4(1+\vert y\vert ^{\beta+1})\vert t_1-t_2\vert^\alpha,}
\end{equation*}
for all $x,y\in\mathbb{R}^d$\textcolor{blue}{, any $t\in(t_0,T]$.}
\end{assp}
\begin{assp}\label{A2}
There exist constants $q\geq2$ and $ C_2>0 $ such that
\begin{equation*}
\langle x-y,\mu(t,x)-\mu(t,y)\rangle+\textcolor{blue}{(q-1)}\vert\sigma(t,x)-\sigma(t,y)\vert^2\leq C_2\vert x-y\vert^2,
\end{equation*}
for all $x,y\in \mathbb{R}^d$\textcolor{blue}{, any $t\in(t_0,T]$.}
\end{assp}
\begin{assp}\label{A3}
There exist constants $p\geq2$ and $C_3>0$ such that 
\begin{equation*}
\langle y,\mu(t,y)\rangle+(p-1)\vert\sigma(t,y)\vert^2\leq C_3(1+\vert y\vert^2),
\end{equation*}
\textcolor{blue}{where $C_3$ does not depend on $t$ or $y$, for all $t\in[t_0,T]$, any $y\in\RR^{d}$.}
\end{assp}
\textcolor{blue}{Assumptions \ref{A1} and \ref{A2} guarantee a unique global solution of SDE \eqref{SDE}.}  In addition, we can derive the boundedness of the moment of the \textcolor{blue}{true} solution from Assumption \ref{A2} \textcolor{blue}{which is proved in \cite{Mao2007}}, that is, there exists a constant $M_1$, which is dependent on $t$ and $q$, such that
\begin{equation}\label{YJ}
E\vert y(t)\vert^q\leq M_1(1+\vert y(0)\vert^q).
\end{equation} \\
And it can be observed from Assumption \ref{A1} that for all $y\in\mathbb{R}^d$ and $t\in[t_0,T]$
\begin{equation}\label{JDZ}
\vert\mu(t,y)\vert\vee\vert\sigma(t,y)\vert\vee\vert L\sigma(t,y)\vert\leq M_2(1+\vert y\vert ^{\beta+1})
\end{equation}
where $M_2$ depends on $C_1$ and $\sup\limits_{t_0\leq t\leq T}\left(\vert\mu(t,0)\vert +\vert\sigma(t,0)\vert\right).$\\
We further assume that there exists a positive constant $M_3$ such that
\begin{equation}\label{DSYJ}
\vert\frac{\partial\mu(t,y)}{\partial y}\vert\vee\vert\frac{\partial^2\mu(t,y)}{\partial y^2}\vert\vee\vert\frac{\partial\sigma(t,y)}{\partial y}\vert\vee\vert\frac{\partial^2\sigma(t,y)}{\partial y^2}\vert\leq M_3(1+\vert y\vert ^{\beta+1}),
\end{equation} \\
\textcolor{blue}{for all $y\in\mathbb{R}^d$ and $t\in[t_0,T]$.}\\
To make the paper self-contained, let us revisit the truncated Milstein method. Firstly, we choose a strictly increasing function $f:\mathbb{R}_+\to\mathbb{R}_+$ \textcolor{blue}{which is the set of all non-negative real numbers} such that $f(u)\to \infty $ as $u\to\infty$ and\\
\begin{equation}\label{supsup}
\sup\limits_{t\in[t_0,T]}\sup\limits_{\vert y\vert\leq u}\left(\vert\mu(t,y)\vert \vee\vert\sigma(t,y)\vert\vee\vert \textcolor{blue}{G^l(t,y)}\vert\right)\leq f(u),\;\forall u\geq1,\textcolor{blue}{\;l=1,2,...,d.}
\end{equation}\\
Then we use $f^{-1}$ denote the inverse function of $f$. we can easily observe that $f^{-1}$ is also a strictly increasing continuous function from $[\textcolor{blue}{f(1)},\infty)$ to $\mathbb{R}_+.$ We choose a constant $\hat{h}\geq1\lor\textcolor{blue}{f(1)}$ and a strictly decreasing function $h:(0,1]\to [\textcolor{blue}{f(1)},\infty)$ such that
\begin{equation}\label{TJ}
\lim_{\Delta\to 0}h(\Delta)=\infty\quad and\quad\Delta^{\frac{1}{4}}h(\Delta)\leq\hat{h},\quad\forall\Delta\in(0,1].
\end{equation}\\
For a given step-size $\Delta\in(0,1]$\textcolor{blue}{, any $t\in[t_0,T]$ and all $y\in\RR^d$},  define the truncated functions by\\
\begin{equation}\label{sigma}
\sigma_{\Delta}(t,y)=\sigma\left(t,\left(\vert y\vert\land f^{-1}\left(h\left(\Delta\right)\right)\right)\frac{y}{\vert y\vert}\right),
\end{equation}
\begin{equation}\label{mu}
\mu_{\Delta}(t,y)=\mu\left(t,\left(\vert y\vert\land f^{-1}\left(h\left(\Delta\right)\right)\right)\frac{y}{\vert y\vert}\right),
\end{equation}
\begin{equation}\label{G}
G_{\Delta}^{\textcolor{blue}{l}}(t,y)=G^{\textcolor{blue}{l}}\left(t,\left(\vert y\vert\land f^{-1}\left(h\left(\Delta\right)\right)\right)\frac{y}{\vert y\vert}\right),\;\textcolor{blue}{l=1,2,...,d},
\end{equation}\\
where we set $y/\vert y\vert=0$ if $y=0$. \\
It is clear that \textcolor{blue}{for all $t\in[t_0,T]$, any $y\in\RR^d$ and $l=1,2,...,d$,}\\
\begin{equation}\label{sigmamuGDelta}
\vert\sigma_{\Delta}(t,y)\vert\vee\vert\mu_{\Delta}(t,y)\vert\vee\vert G_{\Delta}^{\textcolor{blue}{l}}(t,y)\vert\leq f\left(f^{-1}\left(h\left(\Delta\right)\right)\right)=h(\Delta).
\end{equation}\\
We can also obtain the fact that there exists a positive constant $M$ such that 
\begin{equation}\label{Daoshu2}
\vert\frac{\partial \mu_{\Delta}(t,y)}{\partial y}\vert\vee\vert\frac{\partial^2\mu_{\Delta}(t,y)}{\partial y^2}\vert\vee\vert\frac{\partial \sigma_{\Delta}(t,y)}{\partial y}\vert\vee\vert\frac{\partial^2\sigma_{\Delta}(t,y)}{\partial y^2}\vert\leq M,
\end{equation}\\
for any $t\in[t_0,T]$ and $y\in \RR^d$.\\
\begin{rmk}
\textcolor{blue}{Since the truncated functions, $\mu_\Delta(t,y)$, $\s_\Delta(t,y)$ and $\G^l_\Delta$, are not differential at points, $f^{-1}(h(\D))$ and -$f^{-1}(h(\D))$, here we mean the derivatives of the truncated functions  at $f^{-1}(h(\D))$ by their left derivatives and the derivatives at $-f^{-1}(h(\D))$ by their right derivatives.}
\end{rmk}
The discrete-time truncated \textcolor{blue}{M}ilstein numerical solution $X_i$, to approximate $y(t_i)$ for $t_i=i\Delta+t_0$, are formed by setting $X_0=y_0$ and computing 
\begin{equation}\label{discreteMilstein}
\begin{split}
X_{i+1}=X_{i}+\mu_{\D}(t\textcolor{blue}{_i},X_i)\Delta+\sigma_{\Delta}(t\textcolor{blue}{_i},X_i)\Delta B_i+\frac{1}{2}\sum_{l=1}^d\s_{\D}^l(t\textcolor{blue}{_i},X_i){G}_{\D}^l(t\textcolor{blue}{_i},X_i)(\D {B_i}^2-\Delta),
\end{split}
\end{equation}
for $i=0,1,...N,$ where $N$ is the integer part of $T/\Delta$ and let $t_{N+1}=T$ while $\Delta B_i=B(t_{i+1})-B(t_{i})$.\\
To simplify the notation, we set   
\begin{equation*}
L\s_{\D}(t,X_i):=\sum_{l=1}^d\s_{\D}^l(t,X_i)G_{\D}^l(t,X_i)\textcolor{blue}{,}
\end{equation*}\\
\textcolor{blue}{for all $t\in[t_0,T]$ and $l\in\{1,2,...,d\}$}.\\
The continuous version of the truncated \textcolor{blue}{M}ilstein method is defined by
\begin{equation}\label{continuousvertion}
\begin{split}
X\textcolor{blue}{_{\D}}(t)=&\bar{X}\textcolor{blue}{_{\D}}(t)+\int^t_{\textcolor{blue}{t_i}}\textcolor{blue}{\mu_{\D}}\left(\kappa\left(s\right),\bar{X}\textcolor{blue}{_{\D}}\left(s\right)\right)ds+\int_{\textcolor{blue}{t_i}}^t\s_{\D}\left(\kappa\left(s\right),\bar{X}\textcolor{blue}{_{\D}}\left(s\right)\right)dB(s) \\
&+\int_{\textcolor{blue}{t_i}}^tL\s_{\D}\left(\kappa\left(s\right),\bar{X}\textcolor{blue}{_{\D}}\left(s\right)\right)\Delta B(s)dB(s),
\end{split}
\end{equation}
\textcolor{blue}{where $\bar{X}_{\D}(t)$ is a piecewise constant solution such that  $\bar{X}\textcolor{blue}{_{\D}}(t)=\bar{X}_{\D}(t_i)=X_{\D}(t_i)=X_i$, for $t_i\leq t<t_{i+1}$ and we define $\bar{X}_{\D}(T)=X_{\D}(T)$, $\D B(s)=\sum_{i=0}^{N}I_{\{t_i\leq s<t_{i+1}\}}(B(s)-B(t_i))$ and $\kappa(s)=t_iI_{\{t_i\leq s<t_{i+1}\}}$.}
\par
We need the following version of \textcolor{blue}{the} Taylor expansion.
\par
If a function $\phi:\mathbb{R}_+\times\mathbb{R}^d\to\mathbb{R}^d$ is \textcolor{blue}{third-order continuous} differentiable, by \textcolor{blue}{the} Taylor formula we have:
\begin{equation}\label{Secthreedingyi}
\phi(\kappa(t),x)-\phi(\kappa(t),x^*)=\phi'\left(\kappa(t),x\right)\Big|_{x=x^*}(x-x^*)+\textcolor{blue}{R_{\phi}(t,x,x^*)},
\end{equation}
where $\textcolor{blue}{R_{\phi}(t,x,x^*)}=\int_0^1(1-\tau)\textcolor{blue}{\phi''\left(\kappa(t),x\right)\big|_{x=x^*+\tau(x-x^*)}}(x-x^*,x-x^*)d\tau$, for any fixed $t\in[t_0,T].$\\
For any $y,j_1,j_2\in\mathbb{R}^d,$ the expressions of the derivatives are as follows:
\begin{equation*}
\textcolor{blue}{\phi'(\kappa(t),y)(j_1)}=\sum_{i=1}^d\frac{\partial \phi}{\partial {y^i}}j^i_1,\qquad 
\textcolor{blue}{\phi'(\kappa(t),y)(j_1,j_2)}=\sum_{i,j=1}^d\frac{\partial^2 \phi}{\partial {y^i}\partial{y^j}}j^i_1j^j_2,
\end{equation*}
where $\frac{\partial \phi}{\partial {y^i}}=\Big(\frac{\partial \phi_1}{\partial {y^i}},...,\frac{\partial \phi_d}{\partial {y^i}}\Big)$ and $\phi=(\phi_1,\phi_2,...\phi_d)$,  \textcolor{blue}{for any fixed $t\in[t_0,T]$.}\\
If we replace $x$ and $x^*$ by $X\textcolor{blue}{_\D}(t)$ and $\bar{X}\textcolor{blue}{_\D}(t)$ respectively from \eqref{Secthreedingyi}, \textcolor{blue}{for any fixed $t\in[t_0,T]$,} we have:
\begin{equation}
\label{phi}
\begin{split}
\phi(\kappa(t),X\textcolor{blue}{_\D}(t))-\phi(\kappa(t),\bar{X}\textcolor{blue}{_\D}(t))=&\textcolor{blue}{\phi'\left(\kappa(t),x\right)\Big|_{x=\bar{X}_{\D}(t)}}\int_{t_i}^t\s_{\D}(\kappa(s),\bar{X}\textcolor{blue}{_\D}(s))dB(s)\\
&+\textcolor{blue}{\tilde{R}_{\phi}(t,X_{\D}(t),\bar{X}_{\D}(t))}.
\end{split}
\end{equation}
Here,
\begin{equation}\label{secthree2}
\begin{split}
\textcolor{blue}{\tilde{R}_{\phi}(t,X_{\D}(t),\bar{X}_{\D}(t))}=&\textcolor{blue}{\phi'\left(\kappa(t),x\right)\Big|_{x=\bar{X}_{\D}(t)}}\bigg(\int_{t_i}^t\mu_{\D}(\kappa(s),\bar{X}\textcolor{blue}{_\D}(s))ds\\
&+\int_{t_i}^tL\s_{\D}(\kappa(s),\bar{X}\textcolor{blue}{_\D}(s))\Delta B(s)dB(s)\bigg)+\textcolor{blue}{R_{\phi}(t,X_{\D}(t),\bar{X}_{\D}(t))}.
\end{split}
\end{equation}
Thus, replacing $\phi$ by $\sigma_{\Delta}
$ from \eqref{phi}, we obtain
\begin{equation}\label{secthree3}
\textcolor{blue}{R_{\s_{\D}}(t,X_{\D}(t),\bar{X}_{\D}(t))}=\sigma_{\Delta}(\kappa(t),X\textcolor{blue}{_\D}(t))-\sigma_{\Delta}(\kappa(t),\bar{X}\textcolor{blue}{_\D}(t))-L\sigma_{\Delta}(\kappa(t),\bar{X}\textcolor{blue}{_\D}(t))\Delta B(t),
\end{equation}
where $\D B(t)=B(t)-B(t_i),$ for $t_i\leq t< t_{i+1}.$

\section{Main results}
This section is divided into three parts. The main theorems of this paper and the comparison with the existing result are presented in Section 3.1. Important lemmas are proved in Section 3.2. The proofs of the main theorems are postponed to Section 3.3.
\subsection{Main theorem}
\begin{theorem}\label{theo34}
	Let Assumptions \ref{A1}, \ref{A2} \textcolor{blue}{and} \ref{A3} hold and assume that $p>2(1+\beta)q$, then for any $\bar{q}\in [2,q)$ and $\Delta\in(0,1]$\textcolor{blue}{, there exists a constant $H$ such that} 
	\begin{equation*}
	\textcolor{blue}{\sup\limits_{t_0\leq t\leq T}\EE|y(t)-X\textcolor{blue}{_\D}(t)|^{\bar{q}}}
	\leq H\Big(\Delta^{\alpha \bar{q}}+\Delta^{\bar{q}}\big(h(\Delta)\big)^{2\bar{q}}+\Big(f^{-1}\big(h(\Delta)\big)\Big)^{(\beta+1)\bar{q}-p}\Big)
	\end{equation*}
\end{theorem}
To see the convergence rate more clearly, we strength the requirement Assumption \ref{A3} and obtained the following result.
\begin{theorem}\label{th35}
	Let Assumptions \ref{A1} \textcolor{blue}{and} \ref{A2}  hold, and Assumption \ref{A3} hold for any \textcolor{blue}{$p>2(\be+1)q$}. Then for any \textcolor{blue}{$\bar{q}\in[2,q)$},  $\e\in(0,1/4]$ and $\D\in(0,1],$ \textcolor{blue}{there exists a constant $H$ such that}
	\begin{equation*}
	\textcolor{blue}{\sup\limits_{t_0\leq t\leq T}\EE|y(t)-X\textcolor{blue}{_\D}(t)|^{\bar{q}}}\leq H\left(\D^{\min(1-2\e,\a)\bar{q}}\right)
	\end{equation*}
\end{theorem}
\begin{proof}
	First we define $f(u)=H_4u^{\be+2}$, $\h u\geq1.$
	It is easy to get 
	\begin{equation*}
	f^{-1}(u)=\left(\frac{u}{H_4}\right)^{\frac{1}{\be+2}}
	\end{equation*}
	Then, let 
	\begin{equation*}
	h(\D)=\D^{-\e}
	\end{equation*}
	for some $\e\in(0,1/4]$ and $\hat{h}>(1\vee \textcolor{blue}{f(1)})$.\\
	Applying Theorem \ref{theo34}, we can see that 
	\begin{equation*}
	\textcolor{blue}{\sup\limits_{t_0\leq t\leq T}\EE|y(t)-X\textcolor{blue}{_\D}(t)|^{\bar{q}}}\leq H\left(\D^{min\left(\frac{\e\left(p-\left(\be+1\right)\bar{q}\right)}{\be+2},\a\bar{q},\left(1-2\e\right)\bar{q}\right)}\right).
	\end{equation*}
	We can get the desired assertions easily by choosing a sufficiently large $p$. \eproof
\end{proof}
To explain the improvement of the main theorem of this paper, we recall \textcolor{blue}{Theorem 3.7} in \cite{GLMY2018} as follows.
\begin{theorem}\label{thG}
	Let Assumptions \ref{A1}, \ref{A2} and \textcolor{blue}{\eqref{DSYJ}} hold. Furthermore, assume that for any given $p\geq 2$, there exists a $q\in (p,\infty)$. In addition, if
	\begin{equation}
	\label{GLMDELTATJ}
	h(\D)\geq\textcolor{blue}{f}\left(\left(\D^{\textcolor{blue}{\frac{p}{2}}}\left(h(\D)\right)^{{\textcolor{blue}{p}}}\right)^{-1/(q-{\textcolor{blue}{\frac{p}{2}}})}\right)
	\end{equation}
	holds for all sufficiently small $\D\in\textcolor{blue}{(0,1]}$, then for any fixed \textcolor{blue}{$T\geq t_0$}
	\begin{equation}
	\label{GLMSL}
	\EE|\textcolor{blue}{y(T)-X_{N+1}}|^{\textcolor{blue}{p}}\leq K\D^{\textcolor{blue}{p}}\left(h(\D)\right)^{\textcolor{blue}{2p}}
	\end{equation}
	holds, where K is a positive constant independent of $\D$.\
\end{theorem}
\begin{rmk}
Let us demonstrate that compared with the main result in \cite{GLMY2018} our result in this paper releases the constraint on the step-size.
\par
	Consider the scalar SDE
	\begin{equation}
	dy(t)=\left(y(t)-2y^5(t)\right)dt+y^2(t)dB(t),\quad t\geq t_0,
	\end{equation}
	with $t_0=0$  and the  initial value $y(t_0)=1.$\\
Due to the fact that 
\begin{equation*}
\sup\limits_{\vert x\vert\leq u}\left(\vert\mu(x)\vert \vee\vert\sigma(x)\vert\vee\vert L\sigma(x)\vert\right)\leq 3u^5,\quad\forall u\geq1,
\end{equation*}
so we choose $f(u)=3u^5$ and define $h(u)=u^{-\varepsilon}$ for $\varepsilon\in(0,1/4]$.
\par
Choose $\varepsilon=1/4$, then condition \eqref{GLMDELTATJ} is satisfied with $p=1$, $q=12$ and $\D\leq 10^{-21}$. By Theorem \ref{thG} (Theorem 3.7 in \cite{GLMY2018}), we can conclude  that
\begin{equation*}
\EE|y(T)-X_{N+1}|^{2p}\leq K\D^p, 
\end{equation*}
this is to say that the convergence rate is  $1/2$.
\par \noindent
However, take $\varepsilon$ to be $1/4$ and choose $p$ sufficiently large, it can be derived from Theorem \ref{th35} that 
\begin{equation*}
\EE|y(T)-X_\Delta(T)|^{\bar{q}}\leq K\D^{\frac{1}{2}\bar{q}}, 
\end{equation*}
which means that the convergence rate is $1/2$. It should be noted that we do not put the constraint $\D\leq 10^{-21}$ here.
\par
Therefore, compared with the main theorem in \cite{GLMY2018}, the strong requirement on the step-size, $\D\leq 10^{-21}$, is not needed for our main result, which shows that our result releases the  requirement on the step-size.
\end{rmk}
\begin{rmk}\label{mainrmk}
  Theorem \ref{th35} tells us that the order of convergence of the truncated Milstein method is $min(1-2\e,\a)$. If $\alpha$ is close to $1$, the convergence rate  will not very different from that of the traditional truncated Milstein method. Conversely, if $\alpha$ is equal to $1/4$ or more less than $1$, then the order of convergence will worse than the traditional sense. This shows that the H\"older-continuous time variable does affect the order of convergence dramatically.
\end{rmk}
\subsection{Important lemmas}

The next lemma shows that the truncated coefficients inherit the same inequality of the original coefficients.
\begin{lemma}\label{lemma1}
Assume that Assumption \ref{A3}  holds, then for all $\Delta\in(0,1]$, \textcolor{blue}{there exists a positive constant $C_5$ such that}
\begin{equation*}
\langle y,\mu_{\Delta}(t,y)\rangle+(p-1)\vert\s_{\D}(t,y)\vert^2\leq C_5(1+\vert y\vert^2),\quad\forall y\in\mathbb{R}^d\textcolor{blue}{.}
\end{equation*}
\end{lemma}
The proof of this lemma follows the same idea used in \textcolor{blue}{the proof of Lemma 2.5 in} \cite{HLM2018}.
\par
The following lemma shows the difference between the discrete and continuous versions of the truncated Milstein method in the moment sense.
\begin{lemma}\label{Exandexbaryj}
For any $\Delta\in(0,1]$, $t\in[t_0,T]$ and all $p\geq2$\\
\begin{equation*}
\EE\vert X\textcolor{blue}{_\D}(t)-\bar{X}\textcolor{blue}{_\D}(t)\vert^p\leq C\Delta^{\frac {p}{2}}(h(\Delta))^p,
\end{equation*}\\
where $C$ is a constant independent of $\Delta$, consequently,
\begin{equation*}
\lim_{\Delta\to 0}E\vert X\textcolor{blue}{_\D}(t)-\bar{X}\textcolor{blue}{_\D}(t)\vert^p=0,\quad \forall t\in[t_0,T].
\end{equation*}
\end{lemma}
\begin{proof}
Fix the step size $\Delta\in(0,1]$ arbitrarily, for any $t\geq\textcolor{blue}{t_0}$, there exists a constant  $i\geq 0$ such that $t_i\leq t< t_{i+1}.$ We derive from \eqref{continuousvertion} that
\begin{equation*}
\begin{aligned}
\EE\vert X\textcolor{blue}{_\D}(t)-\bar{X}\textcolor{blue}{_\D}(t)\vert^p&\leq C\EE\bigg(\left|\int_{t_i}^t\mu_{\D}\left(\kappa(s),\bar{X}\textcolor{blue}{_\D}(s)\right)ds\right|^p+\left|\int_{t_i}^t\sigma_{\D}\left(\kappa(s),\bar{X}\textcolor{blue}{_\D}(s)\right)dB(s)\right|^p\\
&\quad+\left|\int_{t_i}^tL\s_{\D}\left(\kappa(s),\bar{X}\textcolor{blue}{_\D}(s)\right)\Delta B(s)dB(s)\right|^p\bigg),
\end{aligned}
\end{equation*}
where the elementary inequality $|\sum_{i=1}^m{a_i}|^p\leq m^{p-1}\sum_{i=1}^m|{a_i}|^p$ has been used, and $C$ is a positive constant independent of $\Delta$ that may change from line to line. We  derive from the elementary inequality, the H\"older inequality and \textcolor{blue}{the Burkholder-Davis-Gundy inequality (Theorem 1.7.1 in \cite{Mao2007})}  that
\begin{equation*}
\begin{aligned}
\EE\vert X\textcolor{blue}{_\D}(t)-\bar{X}\textcolor{blue}{_\D}(t)\vert^p&\leq C\bigg(\Delta^{p-1}\EE\int_{t_i}^t\left|\mu_{\D}\left(\kappa(s),\bar{X}\textcolor{blue}{_\D}(s)\right)\right|^pds+\Delta^{\frac{p-2}{2}}\EE\int_{t_i}^t\left|\s_{\D}\left(\kappa(s),\bar{X}\textcolor{blue}{_\D}(s)\right)\right|^pds\\
&\quad+\Delta^{\frac{p-2}{2}}\EE\int_{t_i}^t\left|L\s_{\D}\left(\kappa(s),\bar{X}\textcolor{blue}{_\D}(s)\right)\Delta B(s)\right|^pds\bigg).
\end{aligned}
\end{equation*} \\
By using \eqref{sigmamuGDelta} and the inequality of $\EE|\Delta B(s)|^p\leq C\Delta^{p/2}$ for $s\in[t_i,t_{i+1})$, we get 
\begin{equation*}
\EE\vert X\textcolor{blue}{_\D}(t)-\bar{X}\textcolor{blue}{_\D}(t)\vert^p\leq C\left(\Delta^ph(\Delta)^p+\Delta^{\frac{p}{2}}h(\Delta)^p+\Delta^ph(\Delta)^{2p}\right).
\end{equation*}
Applying \eqref{TJ}, the desired assertion holds.
\eproof
\end{proof}
The next lemma shows the moment boundedness of the truncated Milstein method.
\begin{lemma}\label{shuzhijieYJ}
Let Assumption \ref{A3} holds. Then, for any $\Delta\in(0,1]$ and any $T\geq t_0$\\
\begin{equation}
\sup\limits_{0<\Delta\leq 1}\sup\limits_{t_0\leq t\leq T}\EE\vert X\textcolor{blue}{_\D}(t)\vert^p\leq 
K(1+\EE\vert \textcolor{blue}{y_0}\vert^p),
\end{equation}\\
where $K$ is a positive constant dependent on $T$ but independent of $\Delta$.
\end{lemma}
\begin{proof}
	For any real number $R>X_0$, we define the stopping time 
	\begin{equation*}
	\r_R:=\inf\{t\geq t_0:|X_{\D}(t)|\geq R\}.
	\end{equation*} 
Applying the It\^o formula, we derive from \eqref{continuousvertion}, for any $t\in[t_0,\r_R\land T]$ 
\begin{equation*}
\begin{aligned}
\EE|X\textcolor{blue}{_\D}(t)|^{p}\textcolor{blue}{=}& \EE|X\textcolor{blue}{_\D}(t_0)|^{p}+p\EE\int_{t_0}^t|X\textcolor{blue}{_\D}(s)|^{p-2}\langle X\textcolor{blue}{_\D}(s),\mu_{\D}(\k(s),\bar{X}\textcolor{blue}{_\D}(s))\rangle ds\\
&+\frac{p(p-1)}{2}\EE\int_{t_0}^t|X\textcolor{blue}{_\D}(s)|^{p-2}\Big|\s_{\D}\left(\k(s),\bar{X}\textcolor{blue}{_\D}(s)\right)+L\s_{\D}\left(\k(s),\bar{X}\textcolor{blue}{_\D}(s)\right)\Delta B(s)\Big|^{2}ds,
\end{aligned}
\end{equation*}
where the fact that 
\begin{equation*}
\EE\left(\int_{t_0}^tp|X\textcolor{blue}{_\D}(s)|^{p-2}\left\langle X\textcolor{blue}{_\D}(s),\s_{\D}(\k(s),\bar{X}\textcolor{blue}{_\D}(s))+L\s_{\D}(\k(s),\bar X\textcolor{blue}{_\D}(s))\Delta B(s)\right\rangle dB(s)\right)=0
\end{equation*}
is used.  
\textcolor{blue}{Since $p|X{_\D}(s)|^{p-2}\big\langle X\textcolor{blue}{_\D}(s),\s_{\D}\left(\k(s),\bar{X}\textcolor{blue}{_\D}(s)\right)+L\s_{\D}\left(\k(s),\bar{X}\textcolor{blue}{_\D}(s)\right)\Delta B(s)\big\rangle$ is $\mathcal{F}_s$-measurable, by Theorem 3.2.1 in \cite{Oksendal2003}} we see the fact above is true.
 \\
 We rewrite the inequality as
\begin{equation*}
\begin{aligned}
\EE|X\textcolor{blue}{_\D}(t)|^{p}\leq&
\EE|X\textcolor{blue}{_\D}(t_0)|^{p}+p\EE\int_{t_0}^t|X\textcolor{blue}{_\D}(s)|^{p-2}\left(\left\langle \bar{X}\textcolor{blue}{_\D}(s),\mu_{\D}(\k(s),\bar{X}\textcolor{blue}{_\D}(s))\right\rangle+(p-1)\Big|\s_{\D}\left(\k(s),\bar{X}\textcolor{blue}{_\D}(s)\right)\Big|^{2}\right)\\
&+p(p-1)\EE\int_{t_0}^t|X\textcolor{blue}{_\D}(s)|^{p-2}\Big|L\s_{\D}\left(\k(s),\bar{X}\textcolor{blue}{_\D}(s)\right)\Delta B(s)\Big|^{2}ds\\
&+p\EE\int_{t_0}^t|X\textcolor{blue}{_\D}(s)|^{p-2}\left\langle X\textcolor{blue}{_\D}(s)-\bar{X}\textcolor{blue}{_\D}(s),\mu_{\D}(\k(s),\bar{X}\textcolor{blue}{_\D}(s))\right\rangle ds.
\end{aligned}
\end{equation*}
By Lemma \ref{lemma1}\textcolor{blue}{,} \eqref{sigmamuGDelta} \textcolor{blue}{and Assumption \ref{A3}}, we get
\begin{equation*}
\begin{aligned}
\EE|X\textcolor{blue}{_\D}(t)|^{p}\leq&
\EE|\textcolor{blue}{y_0}|^{p}+K\EE\int_{t_0}^t|X\textcolor{blue}{_\D}(s)|^{p-2}\left(1+|\bar{X}\textcolor{blue}{_\D}(s)|^{2}\right)ds+K\EE\int_{t_0}^t|X\textcolor{blue}{_\D}(s)|^{p-2}|h(\Delta)|^{4}\Delta ds\\
&+p\EE\int_{t_0}^t|X\textcolor{blue}{_\D}(s)|^{p-2}\left\langle X\textcolor{blue}{_\D}(s)-\bar{X}\textcolor{blue}{_\D}(s),\mu_{\D}(\k(s),\bar{X}\textcolor{blue}{_\D}(s))\right\rangle ds, 
\end{aligned}
\end{equation*}
where K is a constant independent of $\Delta$ that may change from line to line.
By using the Young inequality 

\begin{equation*}
a^{p-2}b\leq \frac{p-2}{p}a^{p}+\frac{2}{p}b^{\frac{p}{2}},
\end{equation*}
we have 
\begin{equation}
\begin{aligned}
\EE|X\textcolor{blue}{_\D}(t)|^{p}\leq&
\EE|\textcolor{blue}{y_0}|^{p}+K\EE\int_{t_0}^t|X\textcolor{blue}{_\D}(s)|^{p}ds+K\EE\int_{t_0}^t|\bar{X}\textcolor{blue}{_\D}(s)|^{p}ds\\
&+K\EE\int_{t_0}^t|h(\Delta)|^{2p}\Delta^{\frac{p}{2}}ds+K\EE\int_{t_0}^t |X\textcolor{blue}{_\D}(s)-\bar{X}\textcolor{blue}{_\D}(s)|^{\frac{p}{2}}|\mu_{\D}(\k(s),\bar{X}\textcolor{blue}{_\D}(s))|^{\frac{p}{2}} ds.\label{lemma1zheng1}
\end{aligned}
\end{equation}
By \eqref{TJ}, \eqref{sigmamuGDelta} and Lemma \ref{Exandexbaryj}, we obtain
\begin{equation}\label{lemmazheng2}
\EE\int_{t_0}^t| X\textcolor{blue}{_\D}(s)-\bar{X}\textcolor{blue}{_\D}(s)|^{\frac{p}{2}}|\mu_{\D}(\kappa(s),\bar{X}\textcolor{blue}{_\D}(s))|^{\frac{p}{2}} ds\leq C\int_{t_0}^t(h(\Delta))^{p}\Delta^{\frac{p}{4}}ds\leq C(t-t_{0}).
\end{equation}
Substituting \eqref{lemmazheng2} into \eqref{lemma1zheng1}, by using \eqref{TJ} we see that
\begin{equation*}
\EE|X\textcolor{blue}{_\D}(t)|^{p}\leq
\EE|\textcolor{blue}{y_0}|^{p}+K(t-t_0)+KC(t-t_0)+\textcolor{blue}{K\int_{t_0}^t\left(\sup\limits_{t_0\leq u\leq s}\EE|X_{\D}(u\land\r_R)|^{p}\right)ds}.
\end{equation*}
Under the fact that the sum of the right-hand-side in the above inequality is a increasing function of $t$, we obtain
\begin{equation*}
\sup\limits_{t_0\leq r\leq t}\EE|X\textcolor{blue}{_\D}(r\land \r_R)|^{p}\leq
\EE|\textcolor{blue}{y_0}|^{p}+K(t-t_0)+KC(t-t_0)+\textcolor{blue}{K\int_{t_0}^t\left(\sup\limits_{t_0\leq u\leq s}\EE|X_{\D}(u\land\r_R)|^{p}\right)ds}.
\end{equation*}
\textcolor{blue}{Now the Gronwall inequality yields that
	\begin{equation*}
	\sup\limits_{t_0\leq r\leq T}\EE|X_{\D}(r\land \r_R)|^{p}\leq
	K(1+\EE|y_0|^p).
	\end{equation*}}
\textcolor{blue}{Finally, the desired assertion follows by letting $R\to\infty$.}
\eproof
\end{proof}

\begin{lemma}\label{lemma31}
If Assumptions \ref{A1}, \ref{A3} and \eqref{DSYJ} hold, \textcolor{blue}{and assume that $p\geq2(1+\be)q$, then for any $\bar{q}\in[2,q)$}
\begin{equation*}
\sup\limits_{0<\Delta\leq 1}\sup\limits_{t_0\leq t\leq T}\left[\EE|\mu(t,X\textcolor{blue}{_\D}(t))|^{\textcolor{blue}{2\bar{q}}}\vee \EE|\sigma(t,X\textcolor{blue}{_\D}(t))|^{\textcolor{blue}{2\bar{q}}}\vee \EE\textcolor{blue}{\big|\mu'(t,x)|_{x=X_{\D}(t)}\big|^{2\bar{q}}}\vee \EE\textcolor{blue}{\big|\s'(t,x)|_{x=X_{\D}(t)}\big|^{2\bar{q}}}\right]<\infty.
\end{equation*}
\end{lemma}  
we can derive it from \eqref{DSYJ} \textcolor{blue}{and Lemma \ref{shuzhijieYJ}}.

\begin{lemma}\label{lemma33}
If Assumptions \ref{A1}, \ref{A2}, \ref{A3} and \eqref{DSYJ} hold and assume that $p\geq2(1+\be)q$, then for any $\bar{q}\in(2,q)$ and $\D \in(0,1]$,
\begin{equation*}
\EE\big|\textcolor{blue}{\tilde{R}_{\mu}(t,X_{\D}(t),\bar{X}_{\D}(t))}\big|^{\bar{q}}\vee \EE\big|\textcolor{blue}{\tilde{R}_{\s}(t,X_{\D}(t),\bar{X}_{\D}(t))}\big|^{\bar{q}} \vee \EE\big|\textcolor{blue}{\tilde{R}_{\s_{\D}}(t,X_{\D}(t),\bar{X}_{\D}(t))}\big|^{\bar{q}}\leq C\Delta^{\bar{q}} \big(h(\Delta)\big)^{2\bar{q}},
\end{equation*}
where C is a positive constant independent of $\Delta$.
\end{lemma}
\begin{proof}
Firstly, we give an estimate on $\big|\textcolor{blue}{R_{\mu}(t,X_{\D}(t),\bar{X}_{\D}(t))}\big|^{\bar{q}}$, by Lemmas \ref{Exandexbaryj}\textcolor{blue}{,} \ref{shuzhijieYJ} and \textcolor{blue}{\eqref{DSYJ}}, we obtain a constant C such that
\begin{equation}\label{lemma3zhengming'}
\begin{aligned}
&\EE\big|\textcolor{blue}{R_{\mu}(t,X_{\D}(t),\bar{X}_{\D}(t))}\big|^{\bar{q}}\\
\leq& \int_0^1(1-\tau)^{\bar{q}}\EE\left|\textcolor{blue}{\mu''\left(\kappa(t),x\right)|_{x=\bar{X}_{\D}(t)+\tau\left(X_{\D}(t)-\bar{X}_{\D}(t)\right)}}\left(X\textcolor{blue}{_\D}(t)-\bar{X}\textcolor{blue}{_\D}(t),X\textcolor{blue}{_\D}(t)-\bar{X}\textcolor{blue}{_\D}(t)\right)\right|^{\bar{q}}d\tau\\
\leq&\int_0^1\left[\EE\left|\textcolor{blue}{\mu''\left(\kappa(t),x\right)|_{x=\bar{X}_{\D}(t)+\tau\left(X_{\D}(t)-\bar{X}_{\D}(t)\right)}}\right|^{2\bar{q}}\EE\left|X\textcolor{blue}{_\D}(t)-\bar{X}\textcolor{blue}{_\D}(t)\right|^{4\bar{q}}\right]^{\frac{1}{2}}d\tau\\
\leq&C\Big(1+\EE\big|X\textcolor{blue}{_\D}(t)\big|^{2(1+\be)\bar{q}}+\EE\big|\bar{X}\textcolor{blue}{_\D}(t)\big|^{2(1+\be)\bar{q}}\Big)^{\frac{1}{2}}\left(\EE\left|X\textcolor{blue}{_\D}(t)-\bar{X}\textcolor{blue}{_\D}(t)\right|^{4\bar{q}}\right)^{\frac{1}{2}}\\
\leq& C\Delta^{\bar{q}}h(\Delta)^{2\bar{q}}
\end{aligned}
\end{equation}
where the H\"older inequality  and  the Jensen's inequality are used.\\
Then we can observe from \eqref{secthree2} \textcolor{blue}{and the H\"older inequality} that
\begin{equation}\label{lemma33zheng1'}
\begin{aligned}
\EE\big|\textcolor{blue}{\tilde{R}_{\mu}(t,X_{\D}(t),\bar{X}_{\D}(t))}\big|^{\bar{q}}\leq&C\bigg[\Delta^{\bar{q}}\EE\left|\textcolor{blue}{\mu'\left(\kappa(t),x\right)\big|_{x=\bar{X}_{\D}(t)}}\mu_{\Delta}\left(\kappa(t),\bar{X}\textcolor{blue}{_\D}(t)\right)\right|^{\bar{q}}\\
&+\frac{1}{2}\EE\left|\textcolor{blue}{\mu'\left(\kappa(t),x\right)\big|_{x=\bar{X}_{\D}(t)}}L\sigma_{\Delta}\left(\kappa(t),\bar{X}\textcolor{blue}{_\D}(t)\right)\left(\Delta B(t)^2-\Delta\right)\right|^{\bar{q}}\\
&+\EE\big|\textcolor{blue}{R_{\mu}(t,X_{\D}(t),\bar{X}_{\D}(t))}\big|^{\bar{q}}\bigg]\\
\textcolor{blue}{\leq}&\textcolor{blue}{C\bigg[\Delta^{\bar{q}}\EE\left|\mu'\left(\kappa(t),x\right)\big|_{x=\bar{X}_{\D}(t)}\mu_{\Delta}\left(\kappa(t),\bar{X}_\D(t)\right)\right|^{\bar{q}}}\\
&\textcolor{blue}{+\frac{1}{2}\left(\EE\left|\mu'\left(\kappa(t),x\right)\big|_{x=\bar{X}_{\D}(t)}L\sigma_{\Delta}\left(\kappa(t),\bar{X}_\D(t)\right)\right|^{2\bar{q}}\EE\left|\Delta B(t)^2-\Delta\right|^{2\bar{q}}\right)^{\frac{1}{2}}}\\
&\textcolor{blue}{+\EE\big|\textcolor{blue}{R_{\mu}(t,X_{\D}(t),\bar{X}_{\D}(t))}\big|^{\bar{q}}\bigg]}
\end{aligned}
\end{equation}
for $t_i\leq t< t_{i+1}.$\\
We can derive from the H\"older inequality that
\begin{equation}\label{lemma33zheng2'}
\EE\left|\Delta B(t)^2-\Delta\right|^{\textcolor{blue}{2\bar{q}}}\leq 2^{\textcolor{blue}{2\bar{q}}-1}\left(\EE\big|\Delta B(t)\big|^{\textcolor{blue}{4\bar{q}}}+\Delta^{\textcolor{blue}{2\bar{q}}}\right)\leq 2^{\textcolor{blue}{2\bar{q}}-1}\big(\Delta^{\textcolor{blue}{2\bar{q}}}+\Delta^{\textcolor{blue}{2\bar{q}}}\big)\leq 2^{\textcolor{blue}{2\bar{q}}}\Delta^{\textcolor{blue}{2\bar{q}}}.
\end{equation}
By using Lemma \ref{lemma31} \textcolor{blue}{and} \eqref{sigmamuGDelta}, we can see that for $t_0\leq t\leq T$,
\begin{equation}\label{lemma33zheng31}
\begin{aligned}
\EE\left|\textcolor{blue}{\mu'\left(\kappa(t),x\right)\big|_{x=\bar{X}_{\D}(t)}}\mu_{\Delta}\left(\kappa(t),\bar{X}\textcolor{blue}{_\D}(t)\right)\right|^{\bar{q}}&\leq\textcolor{blue}{\left(h(\D)\right)^{\bar{q}}\EE\left|\textcolor{blue}{\mu'\left(\kappa(t),x\right)\big|_{x=\bar{X}_{\D}(t)}}\right|^{\bar{q}}}\\
&\leq C\big(h(\Delta)\big)^{\bar{q}},
\end{aligned}
\end{equation}
\begin{equation}\label{lemma33zheng32}
\begin{aligned}
\EE\left|\textcolor{blue}{\mu'\left(\kappa(t),x\right)\big|_{x=\bar{X}_{\D}(t)}}L\s_{\Delta}\left(\kappa(t),\bar{X}\textcolor{blue}{_\D}(t)\right)\right|^{\textcolor{blue}{2\bar{q}}}&\leq\textcolor{blue}{\left(h(\D)\right)^{4\bar{q}}\EE\left|\textcolor{blue}{\mu'\left(\kappa(t),x\right)\big|_{x=\bar{X}_{\D}(t)}}\right|^{2\bar{q}}}\\
&\leq C \big(h(\Delta)\big)^{\textcolor{blue}{4\bar{q}}}.
\end{aligned}
\end{equation}
Substituting 
\eqref{lemma3zhengming'}, \eqref{lemma33zheng2'}, \eqref{lemma33zheng31} and \eqref{lemma33zheng32} into \eqref{lemma33zheng1'} and using the independence between $\bar{X}(t)$ and $\Delta B(t)$, we have 
\begin{equation*}
\EE\big|\textcolor{blue}{\tilde{R}_{\mu}(t,X_{\D}(t),\bar{X}_{\D}(t))}\big|^{\bar{q}}\leq C\Delta^{\bar{q}} \big(h(\Delta)\big)^{2\bar{q}}.
\end{equation*}
We obtain the desired result.\\
Similarly, we can show 
\begin{equation*}
 \EE\big|\textcolor{blue}{\tilde{R}_{\s}(t,X_{\D}(t),\bar{X}_{\D}(t))}\big|^{\bar{q}} \vee \EE\big|\textcolor{blue}{\tilde{R}_{\s_{\D}}(t,X_{\D}(t),\bar{X}_{\D}(t))}\big|^{\bar{q}}\leq C\Delta^{\bar{q}} \big(h(\Delta)\big)^{2\bar{q}}.
\end{equation*}
The proof is complete.\eproof
\end{proof}
\subsection{Proof of Theorem 3.1}
\begin{proof}
 Fix $\bar{q}\in[2,q)$ and $\Delta\in(0,1]$ arbitrarily, let $e(t)=y(t)-X(t)$ for $t>t_0$, we define the stopping time for each integer $n>|X_0|$
\begin{equation*}
\theta_n=\inf\Big\{t\geq t_0: \big|X(t)\big|\vee\big|y(t)\big|\geq n\Big\}.
\end{equation*}
We can derive from \textcolor{blue}{the} It\^o formula that for any $t_0\leq t\leq T$,
\begin{equation}\label{th3zheng1}
\begin{aligned}
\EE|e(t\land \theta_n)|^{\bar{q}}\textcolor{blue}{=}& \bar{q}\EE\int_{t_0}^{t\land\theta_n}\big|e(s)\big|^{\bar{q}-2}\Big< y(s)-X\textcolor{blue}{_\D}(s),\mu\big(s,y(s)\big)-\mu_{\Delta}\big(\k(s),\bar{X}\textcolor{blue}{_\D}(s)\big)\Big> ds\\
&+\bar{q}\EE\int_{t_0}^{t\land\theta_n}\frac{\bar{q}-1}{2}\big|e(s)\big|^{\bar{q}-2}\bigg|\sigma\big(s,y(s)\big)-\sigma_{\Delta}\left(\kappa(s),\bar{X}\textcolor{blue}{_\D}(s)\right)\\
&-L\sigma_{\Delta}\left(\kappa(s),\bar{X}\textcolor{blue}{_\D}(s)\right)\Delta B(s)\bigg|^{2}ds.
\end{aligned}
\end{equation}
Substituting \eqref{secthree3} into \eqref{th3zheng1}, we have
\begin{equation*}
\begin{aligned}
\EE|e(t\land \theta_n)|^{\bar{q}}\leq& \bar{q}\EE\int_{t_0}^{t\land\theta_n}\big|e(s)\big|^{\bar{q}-2}\Big< y(s)-X\textcolor{blue}{_\D}(s),\mu\big(s,y(s)\big)-\mu_{\Delta}\big(\k(s),\bar{X}\textcolor{blue}{_\D}(s)\big)\Big> ds\\
&+\bar{q}\EE\int_{t_0}^{t\land\theta_n}\frac{\bar{q}-1}{2}\big|e(s)\big|^{\bar{q}-2}\bigg|\sigma\big(s,y(s)\big)-\sigma_{\Delta}\big(\kappa(s),X\textcolor{blue}{_\D}(s)\big)+\textcolor{blue}{\tilde{R}_{\s_{\D}}(t,X_{\D}(t),\bar{X}_{\D}(t))}
\bigg|^{2}ds.
\end{aligned}
\end{equation*}
By the Young inequality $2ab\leq \varepsilon a^{2}+b^{2}/\varepsilon$ for any $a,b\geq 0$ and  $\varepsilon$ arbitrarily, we choose $\varepsilon=\frac {\textcolor{blue}{q-\bar{q}}}{\textcolor{blue}{\bar{q}-1}}$ here.
\begin{equation*}
\begin{aligned}
&(\bar{q}-1)\Big|\sigma\big(s,y(s)\big)-\sigma_{\Delta}\big(\kappa(s),X\textcolor{blue}{_\D}(s)\big)\Big|^{2}\\
=&(\bar{q}-1)\Big|\sigma\big(s,y(s)\big)-\sigma\big(s,X\textcolor{blue}{_\D}(s)\big)+\sigma\big(s,X\textcolor{blue}{_\D}(s)\big)-\sigma_{\Delta}\big(\kappa(s),X\textcolor{blue}{_\D}(s)\big)\Big|^{2}\\
\leq&(\bar{q}-1)\bigg[\Big(1+\frac {\textcolor{blue}{q-\bar{q}}}{\textcolor{blue}{\bar{q}-1}}\Big)\Big|\sigma\big(s,y(s)\big)-\sigma\big(s,X\textcolor{blue}{_\D}(s)\big)\Big|^{2}\\
\quad&+\Big(1+\frac {\textcolor{blue}{\bar{q}-1}}{\textcolor{blue}{q-\bar{q}}}\Big)\Big|\sigma\big(s,X\textcolor{blue}{_\D}(s)\big)-\sigma_{\Delta}\big(\kappa(s),X\textcolor{blue}{_\D}(s)\big)\Big|^{2}\bigg]\\
=&\textcolor{blue}{(q-1)}\Big|\sigma\big(s,y(s)\big)-\sigma\big(s,X\textcolor{blue}{_\D}(s)\big)\Big|^{2}+\frac{(\bar{q}-1)(q-1)}{\textcolor{blue}{q-\bar{q}}}\Big|\sigma\big(s,X\textcolor{blue}{_\D}(s)\big)-\sigma_{\Delta}\big(\kappa(s),X\textcolor{blue}{_\D}(s)\big)\Big|^{2}.
\end{aligned}
\end{equation*}
Then\\
\begin{equation}\label{thzheng2}
\EE|e(t\land \theta_n)|^{\bar{q}}=J_1+J_2+J_3,
\end{equation}
where
\begin{equation*}
J_1=\EE\int_{t_0}^{t\land\theta_n}\bar{q}\big|e(s)\big|^{\bar{q}-2}\bigg(\Big< e(s),\mu\big(s,y(s)\big)-\mu\big(s,X\textcolor{blue}{_\D}(s)\big)\Big>+\textcolor{blue}{(q-1)}\Big|\sigma\big(s,y(s)\big)-\sigma\big(s,X\textcolor{blue}{_\D}(s)\big)\Big|^{2}\bigg)ds,
\end{equation*}
\begin{equation*}
\begin{aligned}
J_2=&
\EE\int_{t_0}^{t\land\theta_n}\bar{q}\big|e(s)\big|^{\bar{q}-2}\bigg(\Big< e(s),\mu\big(s,X\textcolor{blue}{_\D}(s)\big)-\mu_{\Delta}\big(\k(s),\bar{X}\textcolor{blue}{_\D}(s)\big)\Big>\\
&+\textcolor{blue}{\frac{(\bar{q}-1)(q-1)}{q-\bar{q}}}\Big|\sigma\big(s,X\textcolor{blue}{_\D}(s)\big)-\sigma_{\Delta}\big(\kappa(s),X\textcolor{blue}{_\D}(s)\big)\Big|^{2}\bigg)ds,
\end{aligned}
\end{equation*}
\begin{equation*}
J_3\leq \EE\int_{t_0}^{t\land\theta_n}\bar{q}(\bar{q}-1)\big|e(s)\big|^{\bar{q}-2}\big|\textcolor{blue}{\tilde{R}_{\s_{\D}}(t,X_{\D}(t),\bar{X}_{\D}(t))}\big|^{2}ds.
\end{equation*}
By Assumption \ref{A2}, we have
\begin{equation}\label{th34zheng3}
J_1\leq H_1\EE\int_{t_0}^{t\land\theta_n}\big|e(s)\big|^{\bar{q}}ds.
\end{equation}
Rearranging $J_2$, we get
\begin{equation}\label{th34zheng4}
\begin{aligned}
J_2\leq& \EE\int_{t_0}^{t\land\theta_n}\bar{q}\big|e(s)\big|^{\bar{q}-2}\bigg(\Big<e(s),\mu\big(s,X\textcolor{blue}{_\D}(s)\big)-\mu\big(\kappa(s),X\textcolor{blue}{_\D}(s)\big)\Big>\\
\quad&+\textcolor{blue}{\frac{2(\bar{q}-1)(q-1)}{q-\bar{q}}}\Big|\sigma\big(s,X\textcolor{blue}{_\D}(s)\big)-\sigma\big(\kappa(s),X\textcolor{blue}{_\D}(s)\big)\Big|^{2}\bigg)ds\\  
\quad&+\EE\int_{t_0}^{t\land\theta_n}\bar{q}\big|e(s)\big|^{\bar{q}-2}\bigg(\Big<e(s),\mu\big(\kappa(s),X\textcolor{blue}{_\D}(s)\big)-\mu_{\Delta}\big(\kappa(s),\bar{X}\textcolor{blue}{_\D}(s)\big)\Big>\\ 
\quad&+\textcolor{blue}{\frac{2(\bar{q}-1)(q-1)}{q-\bar{q}}}\Big|\sigma\big(\kappa(s),X\textcolor{blue}{_\D}(s)\big)-\sigma_{\Delta}\big(\kappa(s),X\textcolor{blue}{_\D}(s)\big)\Big|^{2}\bigg)ds\\
=:&J_{21}+J_{22}.
\end{aligned}
\end{equation}
We estimate $J_{21}$ first. Appling the Young inequality $a^{p-2}b^{2}\leq (p-2)a^p/p+2b^{\textcolor{blue}{p}}/p$ for any $a,b\geq0$ and $t_0\leq t\land \theta_n\leq t\leq T$, we obtain
\begin{equation}\label{th34zheng5}
\begin{aligned}
J_{21}\leq& \EE\int_{t_0}^{t\land \theta_n}\bar{q}\big|e(s)\big|^{\bar{q}-2}\bigg(\frac{1}{2}\big|e(s)\big|^{2}+\frac{1}{2}\Big|\mu\big(s,X\textcolor{blue}{_\D}(s)\big)-\mu\big(\kappa(s),X\textcolor{blue}{_\D}(s)\big)\Big|^{2}\\ 
&+\textcolor{blue}{\frac{2(\bar{q}-1)(q-1)}{q-\bar{q}}}\Big|\sigma\big(s,X\textcolor{blue}{_\D}(s)\big)-\sigma\big(\kappa(s),X\textcolor{blue}{_\D}(s)\big)\Big|^{2}\bigg)ds\\ \
\leq& H_2\bigg(\EE\int_{t_0}^{t\land\theta_n}\big|e(s)\big|^{\bar{q}}ds+\EE\int_{t_0}^{t\land\theta_n}\Big|\mu\big(s,X\textcolor{blue}{_\D}(s)\big)-\mu\big(\kappa(s),X\textcolor{blue}{_\D}(s)\big)\Big|^{\bar{q}}ds\\ 
&+\EE\int_{t_0}^{t\land\theta_n}\Big|\sigma\big(s,X\textcolor{blue}{_\D}(s)\big)-\sigma\big(\kappa(s),X\textcolor{blue}{_\D}(s)\big)\Big|^{\bar{q}}ds\bigg)\\  
\leq&H_2\bigg(\EE\int_{t_0}^{t\land\theta_n}\big|e(s)\big|^{\bar{q}}ds+2C_4\EE\int_{t_0}^{\textcolor{blue}{T}}\Big(1+\big|X\textcolor{blue}{_\D}(s)\big|^{(1+\beta)\bar{q}}\Big)\Delta^{\a\bar{q}}ds\bigg)\\
\leq&\textcolor{blue}{H_2\bigg(\EE\int_{t_0}^{t\land\theta_n}\big|e(s)\big|^{\bar{q}}ds+\int_{t_0}^{\textcolor{blue}{T}}\Big(1+\EE\big|X\textcolor{blue}{_\D}(s)\big|^{(1+\beta)\bar{q}}\Big)\Delta^{\a\bar{q}}ds\bigg)}\\
\leq&H_2\bigg(\EE\int_{t_0}^{t\land\theta_n}\big|e(s)\big|^{\bar{q}}ds+\Delta^{\alpha\bar{q}}\bigg),
\end{aligned}
\end{equation}
where the Assumption \textcolor{blue}{\ref{A1}} and Lemma \ref{shuzhijieYJ} are used. Rearranging $J_{22}$ shows that
\begin{equation}\label{J22guji}
\begin{aligned}
J_{22}
\textcolor{blue}{=}&\EE\int_{t_0}^{t\land\theta_n}\bar{q}|e(s)|^{\bar{q}-2}\Big<e(s),\mu\big(\kappa(s),X\textcolor{blue}{_\D}(s)\big)-\mu\big(\kappa(s),\bar{X}\textcolor{blue}{_\D}(s)\big)\Big>ds\\
&+\EE\int_{t_0}^{t\land\theta_n}\bar{q}\big|e(s)\big|^{\bar{q}-2}\bigg(\Big<e(s),\mu\big(\kappa(s),\bar{X}\textcolor{blue}{_\D}(s)\big)-\mu_{\Delta}\big(\kappa(s),\bar{X}\textcolor{blue}{_\D}(s)\big)\Big>\\
&+\textcolor{blue}{\frac{2(\bar{q}-1)(q-1)}{q-\bar{q}}}\Big|\sigma\big(\kappa(s),X\textcolor{blue}{_\D}(s)\big)-\sigma_{\Delta}\big(\kappa(s),X\textcolor{blue}{_\D}(s)\big)\Big|^{2}\bigg)ds\\
=&I_1+I_2, 
\end{aligned}
\end{equation}
where
\begin{equation*}
\begin{aligned}
I_1=&\EE\int_{t_0}^{t\land\theta_n}\bar{q}\big|e(s)\big|^{\bar{q}-2}\Big<e(s),\mu\big(\kappa(s),X\textcolor{blue}{_\D}(s)\big)-\mu\big(\kappa(s),\bar{X}\textcolor{blue}{_\D}(s)\big)\Big>ds,\\
I_2=&\EE\int_{t_0}^{t\land\theta_n}\bar{q}\big|e(s)\big|^{\bar{q}-2}\bigg(\Big<e(s),\mu\big(\kappa(s),\bar{X}\textcolor{blue}{_\D}(s)\big)-\mu_{\Delta}\big(\kappa(s),\bar{X}\textcolor{blue}{_\D}(s)\big)\Big>\\
&+\textcolor{blue}{\frac{2(\bar{q}-1)(q-1)}{q-\bar{q}}}\Big|\sigma\big(\kappa(s),X\textcolor{blue}{_\D}(s)\big)-\sigma_{\Delta}\big(\kappa(s),X\textcolor{blue}{_\D}(s)\big)\Big|^{2}\bigg)ds.
\end{aligned}
\end{equation*}
We can derive from the Young inequality and \textcolor{blue}{\eqref{phi}} that
\begin{equation}\label{th34zheng}
\begin{aligned}
I_1\textcolor{blue}{=}& \EE\int_{t_0}^{t\land\theta_n}\bar{q}\big|e(s)\big|^{\bar{q}-2}\Big<e(s),\textcolor{blue}{\mu'(\kappa(s),x)\big|_{x=\bar{X}_{\D}(s)}}\int_{t_0}^{s}\s_{\D}\big(\kappa(s_1),\bar{X}\textcolor{blue}{_\D}(s_1)\big)dB(s_1)+\textcolor{blue}{\tilde{R}_{\mu}(t,X_{\D}(t),\bar{X}_{\D}(t))}\Big>ds\\ 
\leq&H_{21}\EE\int_{t_0}^{t\land\theta_n}\bigg(\big|e(s)\big|^{\bar{q}}+\Big|e(s)^{T}\textcolor{blue}{\mu'(\kappa(s),x)\big|_{x=\bar{X}_{\D}(s)}}\int_{t_0}^{s}\s_{\D}\big(\kappa(s_1),\bar{X}\textcolor{blue}{_\D}(s_1)\big)dB(s_1)\Big|^{\frac{\bar{q}}{2}}\\
&+\big|e(s)^{T}\textcolor{blue}{\tilde{R}_{\mu}(t,X_{\D}(t),\bar{X}_{\D}(t))}\big|^{\frac{\bar{q}}{2}}\bigg)ds\\
\leq&\bigg(H_{21}\EE\int_{t_0}^{t\land\theta_n}\Big(\big|e(s)\big|^{\bar{q}}+\big|\textcolor{blue}{\tilde{R}_{\mu}(t,X_{\D}(t),\bar{X}_{\D}(t))}\big|^{\bar{q}}\Big)ds\textcolor{blue}{+I_{11}}\bigg),
\end{aligned}
\end{equation}
where 
\begin{equation*}
I_{11}:=\EE\int_{t_0}^{t\land\theta_n}\Big|e(s)^{T}\textcolor{blue}{\mu'(\kappa(s),x)\big|_{x=\bar{X}_{\D}(s)}}\int_{t_0}^{s}\s_{\D}\big(\kappa(s_1),\bar{X}\textcolor{blue}{_\D}(s_1)\big)dB(s_1)\Big|^{\frac{\bar{q}}{2}}ds.
\end{equation*}
Following a very similar approach used for (3.35) in \cite{WG2013}, we get 
\begin{equation}\label{th34zheng6}
I_{11}\leq H_{21}\Delta^{\bar{q}}.
\end{equation}
Combining \eqref{th34zheng}, \eqref{th34zheng6} and Lemma \ref{lemma33}, we obtain
\begin{equation}\label{th34zheng7}
\begin{aligned}
I_1\leq&\textcolor{blue}{H_{21}\bigg(\EE\int_{t_0}^{t\land\theta_n}\big|e(s)\big|^{\bar{q}}ds+\EE\int_{t_0}^{T}\big|\tilde{R}_{\mu}(t,X_{\D}(t),\bar{X}_{\D}(t))\big|^{\bar{q}}ds+\Delta^{\bar{q}}\bigg)}\\
\leq&\textcolor{blue}{H_{21}\bigg(\EE\int_{t_0}^{t\land\theta_n}\big|e(s)\big|^{\bar{q}}ds+\int_{t_0}^{T}\EE\big|\tilde{R}_{\mu}(t,X_{\D}(t),\bar{X}_{\D}(t))\big|^{\bar{q}}ds+\Delta^{\bar{q}}\bigg)}\\ \leq&H_{21}\bigg(\EE\int_{t_0}^{t\land\theta_n}\big|e(s)\big|^{\bar{q}}ds+\Delta^{\bar{q}}\big(h(\Delta)\big)^{2\bar{q}}+\Delta^{\bar{q}}\bigg).
\end{aligned}
\end{equation}
And applying the Young inequality and Assumption \ref{A1}, we can show that
\begin{equation*}
\begin{aligned}
I_2\leq& H_{22}\Bigg(\EE\int_{t_0}^{t\land\theta_n}\big|e(s)\big|^{\bar{q}}ds+\EE\int_{t_0}^{t\land\theta_n}\bigg(\Big|\mu\big(\kappa(s),\bar{X}\textcolor{blue}{_\D}(s)\big)-\mu_{\Delta}\big(\kappa(s),\bar{X}\textcolor{blue}{_\D}(s)\big)\Big|^{\bar{q}}\\
&+\Big|\sigma\big(\kappa(s),X\textcolor{blue}{_\D}(s)\big)-\sigma_{\Delta}\big(\kappa(s),X\textcolor{blue}{_\D}(s)\big)\Big|^{\bar{q}}\bigg)ds\Bigg)\\
\leq& H_{22}\Bigg(\EE\int_{t_0}^{t\land\theta_n}\big|e(s)\big|^{\bar{q}}ds+\EE\int_{t_0}^{t\land\theta_n}\bigg(1+\Big|\bar{X}\textcolor{blue}{_\D}(s)\Big|^{\beta\bar{q}}+\Big|\big|\bar{X}\textcolor{blue}{_\D}(s)\big|\land f^{-1}\big(h(\Delta)\big)\Big|^{\beta\bar{q}}\bigg)\\
&\times\Big|\bar{X}\textcolor{blue}{_\D}(s)-\Big(\big|\bar{X}\textcolor{blue}{_\D}(s)\big|\land f^{-1}\big(h(\Delta)\big)\Big)\frac{\bar{X}\textcolor{blue}{_\D}(s)}{|\bar{X}\textcolor{blue}{_\D}(s)|}\Big|^{\bar{q}} ds+\EE\int_{t_0}^{t\land\theta_n}\bigg(1+\Big|X\textcolor{blue}{_\D}(s)\Big|^{\beta\bar{q}}\\
&+\Big|\big|X\textcolor{blue}{_\D}(s)\big|\land f^{-1}\big(h(\Delta)\big)\Big|^{\beta\bar{q}}\bigg)\Big|X\textcolor{blue}{_\D}(s)-\Big(\big|X\textcolor{blue}{_\D}(s)\big|\land f^{-1}\big(h(\Delta)\big)\Big)\frac{X\textcolor{blue}{_\D}(s)}{|X\textcolor{blue}{_\D}(s)|}\Big|^{\bar{q}}  ds\Bigg)\\
\leq&H_{22}\Bigg(\EE\int_{t_0}^{t\land\theta_n}\big|e(s)\big|^{\bar{q}}ds+\int_{t_0}^{\textcolor{blue}{T}}\Bigg(\EE\bigg[1+\big|\bar{X}\textcolor{blue}{_\D}(s)\big|^{p}+\Big|\big|\bar{X}\textcolor{blue}{_\D}(s)\big|\land f^{-1}\big(h(\Delta)\big)\Big|^{p}\bigg]\Bigg)^{\frac{\be \bar{q}}{p}}\\
&\times\bigg[\EE\Big|\bar{X}\textcolor{blue}{_\D}(s)-\Big(\big|\bar{X}\textcolor{blue}{_\D}(s)\big|\land f^{-1}\big(h(\Delta)\big)\Big)\frac{\bar{X}\textcolor{blue}{_\D}(s)}{|\bar{X}\textcolor{blue}{_\D}(s)|}\Big| ^{\frac{p\bar{q}}{p-\beta \bar{q}}}\bigg]^{\frac{p-\be \bar{q}}{p}}ds+\int_{t_0}^{\textcolor{blue}{T}}\Bigg(\EE\bigg[1+\big|X\textcolor{blue}{_\D}(s)\big|^{p}\\
&+\Big|\big|X\textcolor{blue}{_\D}(s)\big|\land f^{-1}\big(h(\Delta)\big)\Big|^{p}\bigg]
\Bigg)^{\frac{\be\bar{q}}{p}}\bigg[\EE\Big|X\textcolor{blue}{_\D}(s)-\big|X\textcolor{blue}{_\D}(s)\big|\land f^{-1}\big(h(\Delta)\big)\Big|^{\frac{p\bar{q}}{p-\beta \bar{q}}}\bigg]^{\frac{p-\beta \bar{q}}{p}}ds\Bigg)\\
\leq&H_{22}\Bigg(\EE\int_{t_0}^{t\land\theta_n}\big|e(s)\big|^{\bar{q}}ds+\int_{t_0}^{\textcolor{blue}{T}}\bigg(\EE\Big|I\big\{\textcolor{blue}{\big|}\bar{X}\textcolor{blue}{_\D}(s)\textcolor{blue}{\big|}>f^{-1}\big(h(\Delta)\big)\big\}\big|\bar{X}\textcolor{blue}{_\D}(s)\big|^{\frac{p\bar{q}}{p-\beta \bar{q}}}\Big|\bigg)^{\frac{p-\beta \bar{q}}{p}}ds\\
&+\int_{t_0}^{\textcolor{blue}{T}}\bigg(\EE\Big|I\big\{\textcolor{blue}{\big|}X\textcolor{blue}{_\D}(s)\textcolor{blue}{\big|}>f^{-1}\big(h(\Delta)\big)\big\}\big|X\textcolor{blue}{_\D}(s)\big|^{\frac{p\bar{q}}{p-\beta \bar{q}}}\Big|\bigg)^{\frac{p-\beta \bar{q}}{p}}ds\Bigg),
\end{aligned}
\end{equation*}
where the H\"older inequality and Lemma \ref{shuzhijieYJ} are used above, and using the Chebyshev inequality yields
\begin{equation}\label{th34zheng8}
\begin{aligned}
I_2\leq&
H_{22}\Bigg(\EE\int_{t_0}^{t\land\theta_n}\big|e(s)\big|^{\bar{q}}ds+\int_{t_0}^{\textcolor{blue}{T}}\bigg(\Big[P\Big\{\big|\bar{X}\textcolor{blue}{_\D}(s)\big|>f^{-1}\big(h(\Delta)\big)\Big\}\Big]^{\frac{p-\beta \bar{q}-\bar{q}}{p-\beta \bar{q}}}\Big[\EE\big|\bar{X}\textcolor{blue}{_\D}(s)\big|^{p}\Big]^{\frac{\bar{q}}{p-\beta \bar{q}}}\bigg)^{\frac{p-\beta \bar{q}}{p}}ds\\  
+&\int_{t_0}^{\textcolor{blue}{T}}\bigg(\bigg[P\Big\{\big|X\textcolor{blue}{_\D}(s)\big|>f^{-1}\big(h(\Delta)\big)\Big\}\bigg]^{\frac{p-\beta \bar{q}-\bar{q}}{p-\beta \bar{q}}}\Big[\EE\big|X\textcolor{blue}{_\D}(s)\big|^{p}\Big]^{\frac{\bar{q}}{p-\beta \bar{q}}}\bigg)^{\frac{p-\beta \bar{q}}{p}}ds\Bigg)\\  
\leq&H_{22}\Bigg(\EE\int_{t_0}^{t\land\theta_n}\big|e(s)\big|^{\bar{q}}ds+\int_{t_0}^{\textcolor{blue}{T}}\bigg(\frac{\EE\big|\bar{X}\textcolor{blue}{_\D}(s)\big|^{p}}{\big|f^{-1}\big(h(\Delta)\big)\big|^{p}}\bigg)^{\frac{p-\beta \bar{q}-\bar{q}}{p}}ds+\int_{t_0}^{\textcolor{blue}{T}}\bigg(\frac{\EE\big|X\textcolor{blue}{_\D}(s)\big|^{p}}{\big|f^{-1}\big(h(\Delta)\big)\big|^{p}}\bigg)^{\frac{p-\beta \bar{q}-\bar{q}}{p}}ds\Bigg)\\
\leq &H_{22}\bigg(\EE\int_{t_0}^{t\land\theta_n}\big|e(s)\big|^{\bar{q}}ds+\Big(f^{-1}\big(h(\Delta)\big)\Big)^{(\beta+1)\bar{q}-p}\bigg).
\end{aligned}
\end{equation}
Substituting \eqref{th34zheng7} and \eqref{th34zheng8} into \eqref{J22guji} gives
\begin{equation}\label{J22gujizuizhong}
J_{22}\leq H_{2}\bigg(\EE\int_{t_0}^{t\land\theta_n}\big|e(s)\big|^{\bar{q}}ds+\Big(f^{-1}\big(h(\Delta)\big)\Big)^{(\beta+1)\bar{q}-p}+\Delta^{q}\big(h(\Delta)\big)^{2\bar{q}}+\D^{\bar{q}}\bigg).
\end{equation}
Due to the Young inequality and Lemma \ref{lemma33}, we derive that 
\begin{equation}\label{J3guji}
\begin{split}
J_3\leq& H_{3}\EE\int_{t_0}^{t\land\theta_n}\bigg(\big|e(s)\big|^{\bar{q}}+\big|\textcolor{blue}{\tilde{R}_{\s_{\D}}(t,X_{\D}(t),\bar{X}_{\D}(t))}\big|^{\bar{q}}\bigg)ds\\
\leq& H_{3}\bigg(\EE\int_{t_0}^{t\land\theta_n}\big|e(s)\big|^{\bar{q}}ds+\int_{t_0}^{\textcolor{blue}{T}}\EE|\textcolor{blue}{\tilde{R}_{\s_{\D}}(t,X_{\D}(t),\bar{X}_{\D}(t))}\big|^{\bar{q}}ds\bigg)\\
\leq&H_{3}\bigg(\EE\int_{t_0}^{t\land\theta_n}\big|e(s)\big|^{\bar{q}}ds+\Delta^{\bar{q}}\big(h(\Delta)\big)^{2\bar{q}}\bigg),
\end{split}
\end{equation}
where $H_{21},$ $H_{22}$, $H_3$ and  the following $H$ are generic constants independent of $\Delta$ that may change from line to line.\\
Combining \eqref{thzheng2}, \eqref{th34zheng3}, \eqref{th34zheng4}, \eqref{th34zheng5}, \eqref{J22gujizuizhong} and \eqref{J3guji} together, we can see that
\begin{equation*}
\begin{split}
\EE|e(t\land \theta_n)|^{\bar{q}}\leq& H\bigg(\EE\int_{t_0}^{t\land\theta_n}\big|e(s)\big|^{\bar{q}}ds+\Delta^{\alpha \bar{q}}+\Delta^{\bar{q}}\big(h(\Delta)\big)^{2\bar{q}}+\Delta^{\bar{q}}+\Big(f^{-1}\big(h(\Delta)\big)\Big)^{(\beta+1)\bar{q}-p}\bigg)\\
\leq&\textcolor{blue}{H\bigg(\int_{t_0}^t\sup\limits_{t_0\leq u\leq s}\EE \big|e(u\land\theta_n)\big|^{\bar{q}}ds+\Delta^{\alpha \bar{q}}+\Delta^{\bar{q}}\big(h(\Delta)\big)^{2\bar{q}}+\Delta^{\bar{q}}+\Big(f^{-1}\big(h(\Delta)\big)\Big)^{(\beta+1)\bar{q}-p}\bigg).}
\end{split}
\end{equation*}
An application of the Gronwall inequality yields that 
\begin{equation*}
\textcolor{blue}{\sup\limits_{t_0\leq r\leq T}\EE|e(r\land \theta_n)|^{\bar{q}}}\leq H\Big(\Delta^{\alpha \bar{q}}+\Delta^{\bar{q}}\big(h(\Delta)\big)^{2\bar{q}}+\Big(f^{-1}\big(h(\Delta)\big)\Big)^{(\beta+1)\bar{q}-p}\Big).
\end{equation*}
\textcolor{blue}{Due to the existence and uniqueness of the global solution to SDE \eqref{SDE} in $[t_0,T]$, we have $T\land\theta_n\to T$ as $n\to\infty$ (see, for example, the proof of Lemma 2.3.2 in \cite{Mao2007}, where the similar argument was used).}
Using Fatou Lemma and letting $n\rightarrow \infty $, the desired assertion is obtained.
\eproof
\end{proof}

\section{Randomized Truncated Milstein method}
\textcolor{blue}{To define the randomized truncated Milstein method, let $(\tau_j)_{j\in\mathbb{N}}$ be a i.i.d  family of $\mathcal{U}(0,1)$-distributed random variables on a filtered probabilty space $(\omega_{\tau},\mathcal{F}^{\tau},(\mathcal{F}_j^{\tau})_{j\in{\mathbb{N}}},\mathbb{P}_{\tau})$, the space is generated by $\{\tau_1,...,\tau_j\}$. Besides we define $\mathcal{U}(0,1)$ as a unique distribution on the interval $(0,1)$. Furthermore, $(\tau_j)_{j\in\mathbb{N}}$ is assumed to be  independent of the randomness in SDE \eqref{SDE}.}\par
As already observed in Remark \ref{mainrmk}, the convergence rate of the truncated Milstein method is dominated by the H\"older index $\alpha$. The purpose of this section is to propose some new method to improve the convergence rate. 
\par
Inspired by \cite{KW2019}, we embed the randomized time step into \eqref{discreteMilstein} and propose the following randomized truncated Milstein method.\par
\textcolor{blue}{Given  a step-size $\D\in(0,1),$ the randomized truncated Milstein numerical solution $X_{i+1}$ to approximate of SDE \eqref{SDE} for $t_i=i{\D}$ is given by the recursion}
\begin{equation*}
X^{\tau}_{i+1}=X_i+\tau_{i}\D\mu_{\D}(t_i,X_i)+\s_{\D}(t_i,X_i)\textcolor{blue}{\big(B(t_i+\tau_i\D)-B(t_i)\big)},
\end{equation*}
\begin{equation*}
X_{i+1}=X_i+\D\mu_{\D}(t_i+\tau_{i}\D,X_{i+1}^{\tau})+\s_{\D}(t_i,X_i)\D B_i+\frac{1}{2}\sum_{l=1}^d\s_{\D}^l(t_i,X_i)G_{\D}^l(t_i,X_i)(\D B_i^2-\D),
\end{equation*}
\textcolor{blue}{where $X_0=y_0,\;t_{N+1}=T$ for $N$ is the integer part of $T/{\D}$ and  $\D B_i=B(t_{i+1}-B(t_i))$ for $i\in\{0,1,...,N\}$}.
\par
Based on \cite{KW2019}, we have the following conjecture on the convergence rate. Briefly speaking, with the employment of the randomized technique the convergence is improved from $\min(1-2\e, \alpha)$ to $\min(1-2\e, \alpha+1/2)$.
\par
Since we have still been working on the proof of it, we will demonstrate this conjecture by using numerical simulation in the next section.

\begin{conj}\label{theconj}
Suppose Assumptions \ref{A1}, \ref{A2} and \ref{A3}
hold for any $p>2$, then for any $\bar{q}>0,\varepsilon\in(0,1/4)$ and $\D\in(0,1],$ 
	\begin{equation*}
	\textcolor{blue}{\sup\limits_{t_0\leq t\leq T}\EE|y(t)-X\textcolor{blue}{_\D}(t)|^{\bar{q}}}\leq H\left(\D^{min(1-2\e,\alpha+\frac{1}{2})\bar{q}}\right),
	\end{equation*}
{\color{blue} where $H$ is a constant independent from $\Delta$.}
\end{conj}

\section{Numerical examples}
The purpose of the example discussed in this section is twofold. On one side, it is used to illustrate Theorem \ref{th35}. On the other side, it demonstrates that the convergence rate in Conjecture \ref{theconj} is promising. 
\begin{expl}\label{Ex41}
	Consider the scaler SDE
	\begin{equation}
	\left\{
	\begin{array}{lr}
	dX(t)=\left(\left[t\left(1-t\right)\right]^{\frac{1}{4}}X^2(t)-X^5(t)\right)dt+\left[t\left(1-t\right)\right]^{\frac{3}{4}}X(t)dB(t),&\\
	X(t_0)=2,
	\end{array}
	\right.
	\end{equation}
where $t_0=0,$ $T=1$ and $B(t)$ is a scalar Brownian motion.
\end{expl}
For any $q>2$, $t\in[0,1]$ we can see that
\begin{equation*}
\begin{aligned}
&(x-y)\left(\mu\left(t,x\right)-\mu\left(t,y\right)\right)+\textcolor{blue}{(q-1)}\left|\sigma\left(t,x\right)-\sigma\left(t,y\right)\right|^2\\
\textcolor{blue}{=}&(x-y)^2\left(\left[t\left(1-t\right)\right]^{\frac{1}{4}}(x+y)-\left(x^4+x^3y+x^2y^2+xy^3+y^4\right)+\textcolor{blue}{(q-1)}\left[t\left(1-t\right)\right]^{\frac{3}{2}}\right).
\end{aligned}
\end{equation*}
But
\begin{equation*}
-(x^3y+xy^3)=-xy(x^2+y^2)\leq 0.5(x^2+y^2)^2=0.5(x^4+y^4)+x^2y^2.
\end{equation*}
Hence
\begin{equation*}
\begin{aligned}
&(x-y)^T\left(\mu\left(t,x\right)-\mu\left(t,y\right)\right)+\textcolor{blue}{(q-1)}\left|\sigma\left(t,x\right)-\sigma\left(t,y\right)\right|^2\\
\leq&(x-y)^2\left(\left[t\left(1-t\right)\right]^{\frac{1}{4}}(x+y)-0.5\left(x^4+y^4\right)+\textcolor{blue}{(q-1)}\left[t\left(1-t\right)\right]^{\frac{3}{2}}\right)\\
\leq&C(x-y)^2.
\end{aligned}
\end{equation*}
Under the fact that polynomials with negative coefficient for the highest order term can always be bounded, we can obtain the assertion above. It means that Assumption \ref{A2} is satisfied. 
\par
Similarly, for any $p>2$ and any $t\in[0,1]$, we have
\begin{equation*}
x^T\mu(t,x)+(p-1)|\s(t,x)|^2=\left[t\left(1-t\right)\right]^\frac{1}{4}x^3-x^6+(p-1)\left[t\left(1-t\right)\right]^{\frac{3}{2}}x^2\leq C\left(1+\left|x\right|^2\right),
\end{equation*}
which \textcolor{blue}{shows} that Assumption \ref{A3} holds.
\par
Appling the mean value theorem for the temporal variable, Assumption \ref{A1}  are satisfied with $\a=1/4$ and $\beta=4$. 
Due to the fact that 
\begin{equation*}
\sup\limits_{0<t\leq 1}\sup\limits_{|x|\leq u}\left(|\mu(t,x)|\vee|\s(t,x)|\vee|L\s(t,x)|\right)\leq 2u^5,\quad\forall u\geq1,
\end{equation*} 
Let $f(u)=2u^5$ and $h(\D)=\D^{-\e}$, for any $\e\in(0,1/4).$
Choose $\textcolor{blue}{\epsilon}$ sufficiently \textcolor{blue}{small}, we can derive from Theorem \ref{th35} that
\begin{equation*}
\textcolor{blue}{\sup\limits_{0\leq t\leq 1}\EE|y(t)-X\textcolor{blue}{_\D}(t)|^{\bar{q}}}\leq H\D^{\frac{1}{4}\bar{q}},
\end{equation*}
\begin{figure}[H]
	\centering
	\includegraphics[width=0.80\textwidth]{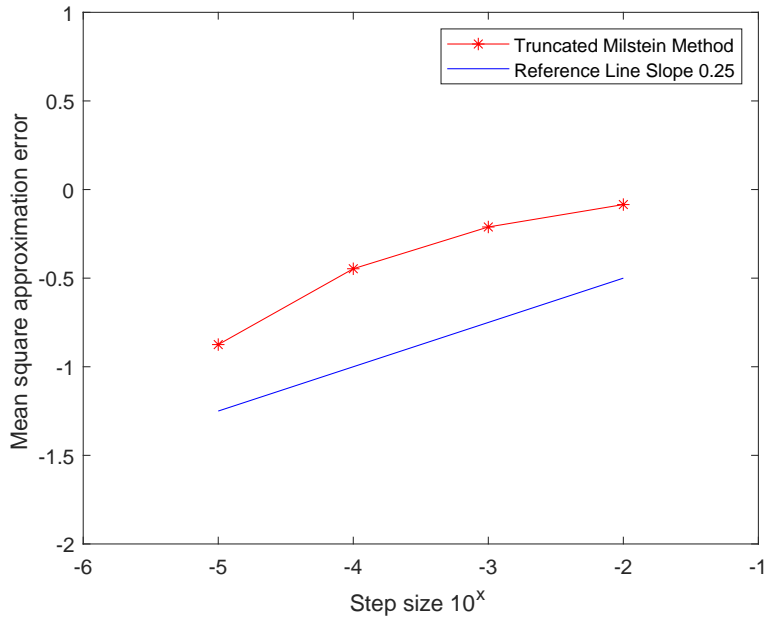}
	\caption{Convergence rate of truncated Milstein method in Example }
	\label{1}
\end{figure}
This shows that the convergence rate of truncated Milstein method for the SDE 
\eqref{Ex41} is $1/4$.
To approximate the mean square error, we run $M=1000$ independent trajectories for 5 different step sizes. And we regard the numerical solution as the step-size $10^{\textcolor{blue}{-7}}$ as the true solution for the SDE. By numerical simulation we can see in Figure \ref{1} that the slope of the error against the step sizes is about \textcolor{blue}{0.2562}. \par 
Let us turn to the discussion on the randomized truncated Milstein method.  We would expect
\begin{equation*}
\sup\limits_{0\leq t\leq 1}\EE|y(t)-{X}\textcolor{blue}{_\D}(t)|^{\bar{q}}\leq H\D^{\frac{3}{4}\bar{q}},
\end{equation*}
\begin{figure}[H]
	\centering
	\includegraphics[width=0.80\textwidth]{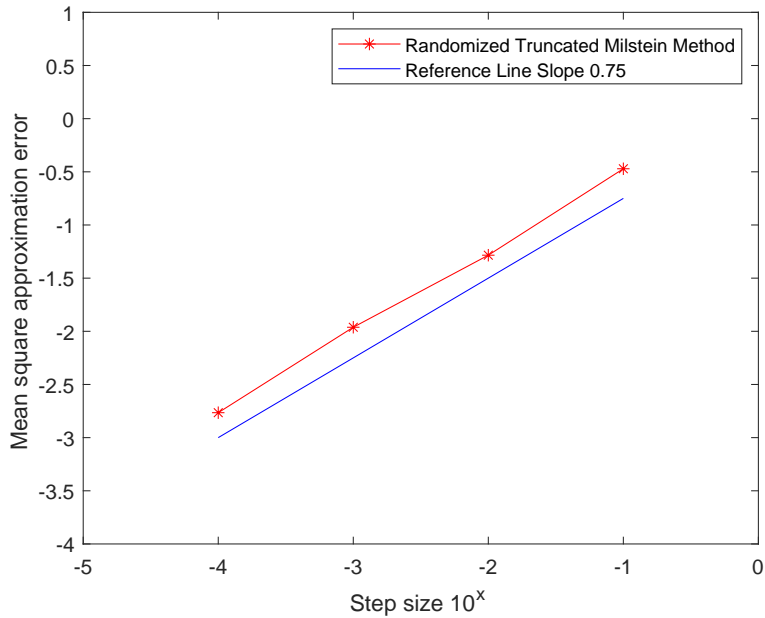}
	\caption{Convergence rate of randomized truncated Milstein method }
	\label{2}
\end{figure}
\textcolor{blue}{In numerical simulations, we use the step-size $10^{-6}$ to approximate the true solution. In Figure \ref{2}, we run M=1000 independent paths with step-sizes $10^{-4}, 10^{-3}, 10^{-2}, 10^{-1}.$ }
\par
It is clearly to see from Figure \ref{2} that the convergence rate of the randomized truncated Milstein method is indeed improved to be $0.7548$. This shows that Conjecture \ref{theconj} is reasonable.

\section{Conclusion and future research}
This paper revisited the truncated Milstein method and proved the strong convergence of the method for non-autonomous SDEs, which extended and improved the existing result.
\par
With the observation that the convergence rate could be very low due to the H\"older continuous time variable, the randomized truncated Milstein method was proposed. The conjecture on the improvement of the convergence rate is reported. Numerical simulations demonstrate the conjecture is promising.
\par
One of the main future works is to prove the conjecture. In addition, we are working on the stability of the truncated Milstein method and the randomized truncated Milstein method in different senses.

\section*{Acknowledgements}
We deeply appreciate the detailed comments and suggestions from the reviewers and the editor, which helped to improve this work significantly.
\par
The authors would like to thank the Natural Science Foundation of China (11701378, 11871343, 11971316), and Science and Technology Innovation Plan of Shanghai (20JC1414200) for their financial support.


\begin{thebibliography}{99}

\bibitem{AH2017}
S. Amiri and S.M. Hosseini, 
Stochastic Runge-Kutta Rosenbrock type methods for SDE systems, 
Applied Numerical Mathematics, \textcolor{blue}{115, 2017, 1-15.}

\bibitem{BK2010}
W.-J. Beyn and R. Kruse, 
Two-sided error estimates for the stochastic theta method, 
Discrete and Continuous Dynamical Systems - Series B, \textcolor{blue}{14(2), 2010, 389-407.
}

\bibitem{BT2004}
K. Burrage and T. Tian, 
Implicit stochastic Runge-Kutta methods for stochastic differential equations, 
BIT Numerical Mathematics, \textcolor{blue}{44(1), 2004, 21-39.
}

\bibitem{Daun2011}
T. Daun, 
On the randomized solution of initial value problems, 
Journal of Complexity, \textcolor{blue}{27(3-4), 2011, 300-311.
}
\bibitem{DR2009}
K. Debrabant and A. R\"{o}\ss ler, 
Diagonally drift-implicit Runge-Kutta methods of weak order one and two for It\^o SDEs and stability analysis, 
Applied Numerical Mathematics, \textcolor{blue}{59(3-4), 2009, 595-607.}

\bibitem{DFLM2019}
S. Deng, W. Fei, W. Liu and X. Mao, 
The truncated EM method for stochastic differential equations with Poisson jumps, 
Journal of Computational and Applied Mathematics, \textcolor{blue}{355, 2019, 232-257.}



\bibitem{DKS2016}
K. Dareiotis, C. Kumar and S. Sabanis, 
On tamed Euler approximations of SDEs driven by L\'evy noise with applications to delay equations, 
SIAM Journal on Numerical Analysis, \textcolor{blue}{54(3), 2016, 1840-1872.}


\bibitem{GHW2020}
S. Gan, Y. He and X. Wang, 
Tamed Runge-Kutta methods for SDEs with super-linearly growing drift and diffusion coefficients, 
Applied Numerical Mathematics, \textcolor{blue}{152, 2020, 379-402.}

\bibitem{GLMY2017}
Q. Guo, W. Liu, X. Mao and R. Yue, 
The partially truncated Euler-Maruyama method and its stability and boundedness, 
Applied Numerical Mathematics, \textcolor{blue}{115, 2017, 235-251.}

\bibitem{GLMY2018}
Q. Guo, W. Liu, X. Mao and R. Yue, 
The truncated Milstein method for stochastic differential equations with commutative noise, 
Journal of Computational and Applied Mathematics, \textcolor{blue}{338, 2018, 298-310.} 


\bibitem{HW2010}
E. Hairer and G. Wanner, 
Solving ordinary differential equations. II. Stiff and differential-algebraic problems. Second revised edition. Springer Series in Computational Mathematics, \textcolor{blue}{14. Springer-Verlag, Berlin, 2010.}


\bibitem{Higham2011}
D.J. Higham, 
Stochastic ordinary differential equations in applied and computational mathematics, 
IMA Journal of Applied Mathematics, \textcolor{blue}{76(3), 2011, 449-474.}

\bibitem{HMS2013}
D.J. Higham, X. Mao and L. Szpruch, 
Convergence, non-negativity and stability of a new Milstein scheme with applications to finance, 
Discrete and Continuous Dynamical Systems - Series B, \textcolor{blue}{18(8), 2013, 2083-2100.}


\bibitem{HLM2018}
L. Hu, X. Li and X. Mao, 
Convergence rate and stability of the truncated Euler-Maruyama method for stochastic differential equations, 
Journal of Computational and Applied Mathematics, \textcolor{blue}{337, 2018, 274-289.}

\bibitem{Hu1996}
Y. Hu, 
	Semi-implicit Euler-Maruyama scheme for stiff stochastic equations. 
	Stochastic analysis and related topics, V (Silivri, 1994), 183-202, Progr. Probab., 38, Birkhäuser Boston, Boston, MA, 1996.




\bibitem{HJK2011}
M. Hutzenthaler, A. Jentzen and P.E. Kloeden, 
Strong and weak divergence in finite time of Euler's method for stochastic differential equations with non-globally Lipschitz continuous coefficients, 
Proceedings of the Royal Society A: Mathematical, Physical and Engineering Sciences, \textcolor{blue}{467(2130), 2011, 1563-1576.}



\bibitem{HJK2012}
M. Hutzenthaler, A. Jentzen and P.E. Kloeden, 
Strong convergence of an explicit numerical method for sdes with nonglobally lipschitz continuous coefficients, 
Annals of Applied Probability, \textcolor{blue}{22(4), 2012, 1611-1641.}


\bibitem{JHL2018}
Y. Jiang, Z. Huang and W. Liu, 
Equivalence of the mean square stability between the partially truncated Euler-Maruyama method and stochastic differential equations with super-linear growing coefficients, 
Advances in Difference Equations,  2018, Article No. 355.


\bibitem{Kacewicz1987}
B.Z. Kacewicz, 
Optimal solution of ordinary differential equations, 
Journal of Complexity, \textcolor{blue}{3(4), 1987, 451-465.}

\bibitem{KS2006}
C. Kahl and H. Schurz, 
Balanced Milstein methods for ordinary SDEs, 
Monte Carlo Methods and Applications, \textcolor{blue}{12(2), 2006, 143-170.}


\bibitem{KW2019}
R. Kruse and Y. Wu, 
A randomized milstein method for stochastic differential equations with non-differentiable drift coefficients, 
Discrete and Continuous Dynamical Systems - Series B, \textcolor{blue}{24(8), 2019, 3475-3502.}


\bibitem{LX2018}
G. Lan and F. Xia, 
Strong convergence rates of modified truncated EM method for stochastic differential equations, 
Journal of Computational and Applied Mathematics, \textcolor{blue}{334, 2018, 1-17.}

\bibitem{LMY2019}
X. Li, X. Mao and G. Yin, 
Explicit numerical approximations for stochastic differential equations in finite and infinite horizons: Truncation methods, convergence in pth moment and stability, 
IMA Journal of Numerical Analysis, \textcolor{blue}{39(2), 2019, 847-892.}


\bibitem{LMTW2020}
W. Liu, X. Mao, J. Tang and Y. Wu, 
Truncated Euler-Maruyama method for classical and time-changed non-autonomous stochastic differential equations, 
Applied Numerical Mathematics, \textcolor{blue}{153, 2020, 66-81.}


\bibitem{Mao2007}
X. Mao, 
Stochastic Differential Equations and Applications, 
2nd Edition, Horwood Publishing,
Chichester, 2007.


\bibitem{Mao2015}
X. Mao, 
The truncated Euler-Maruyama method for stochastic differential equations, 
Journal of Computational and Applied Mathematics, \textcolor{blue}{290, 2015, 370-384.} 


\bibitem{Mao2016}
X. Mao, 
Convergence rates of the truncated Euler-Maruyama method for stochastic differential equations, 
Journal of Computational and Applied Mathematics, \textcolor{blue}{296, 2016, 362-375.}


\bibitem{MS2013}
X. Mao and L. Szpruch, 
Strong convergence rates for backward Euler-Maruyama method for non-linear dissipative-type stochastic differential equations with super-linear diffusion coefficients, 
Stochastics, \textcolor{blue}{85(1), 2013, 144-171}.


\bibitem{NL2019}
H.L. Ngo and D.T. Luong, 
Tamed Euler-Maruyama approximation for stochastic differential equations with locally H\"older continuous diffusion coefficients, 
Statistics and Probability Letters, \textcolor{blue}{145, 2019, 133-140.}

\bibitem{Oksendal2003}
B. Oksendal, Stochastic Differential Equations, 6th Edition, Springer, Berlin, Heidelberg, 2003.

\bibitem{RKVZ2015}
V. Reshniak, A.Q.M. Khaliq, D.A. Voss and G. Zhang, 
Split-step Milstein methods for multi-channel stiff stochastic differential systems, 
Applied Numerical Mathematics, \textcolor{blue}{89, 2015, 1-23.}


\bibitem{SLL2018}
M. Song, Y. Lu and M. Liu, 
Convergence of the Tamed Euler Method for Stochastic Differential Equations with Piecewise Continuous Arguments Under Non-global Lipschitz Continuous Coefficients, 
Numerical Functional Analysis and Optimization, \textcolor{blue}{39(5), 2018, 517-536.}


\bibitem{WG2013}
X. Wang and S. Gan, 
The tamed Milstein method for commutative stochastic differential equations with non-globally Lipschitz continuous coefficients, 
Journal of Difference Equations and Applications, \textcolor{blue}{19(3), 2013, 466-490.}


\bibitem{WWD2020}
X. Wang, J. Wu and B. Dong, 
Mean-square convergence rates of stochastic theta methods for SDEs under a coupled monotonicity condition, 
BIT Numerical Mathematics, \textcolor{blue}{60(3), 2020, 759-790.}


\bibitem{YXT2017}
Z. Yan, A. Xiao and X. Tang, 
Strong convergence of the split-step theta method for neutral stochastic delay differential equations, 
Applied Numerical Mathematics, \textcolor{blue}{120, 2017, 215-232.}


\bibitem{ZM2017}
Z. Zhang and H. Ma, 
Order-preserving strong schemes for SDEs with locally Lipschitz coefficients, 
Applied Numerical Mathematics, \textcolor{blue}{112, 2017, 1-16.}


\bibitem{ZSL2018}
W. Zhang, M. Song and M. Liu, 
Strong convergence of the partially truncated Euler–Maruyama method for a class of stochastic differential delay equations, 
Journal of Computational and Applied Mathematics, \textcolor{blue}{335, 2018, 114-128.}


\bibitem{ZWH2014}
X. Zong, F. Wu and C. Huang, 
Convergence and stability of the semi-tamed Euler scheme for stochastic differential equations with non-Lipschitz continuous coefficients, 
Applied Mathematics and Computation, \textcolor{blue}{228, 2014, 240-250.}


\bibitem{ZWX2018}
X. Zong, F. Wu and G. Xu, 
Convergence and stability of two classes of theta-Milstein schemes for stochastic differential equations, 
Journal of Computational and Applied Mathematics, \textcolor{blue}{336, 2018, 8-29.}




\end{thebibliography}
\end{document}